\documentclass[12pt,a4paper]{article}%
\usepackage{amsmath}
\usepackage{graphicx}
\usepackage{epsfig}%
\usepackage{amsfonts}%
\usepackage{amssymb}

\usepackage{color}

\usepackage[normalem]{ulem}

\newcommand{\tb}[1]{\textbf{#1}}
\renewcommand{\it}[1]{$\mathit{#1}$}

\newcommand{\pc}[1]{\left[ #1 \right]}
\newcommand{\pl}[1]{\left\{ #1 \right \}}

\newtheorem{theorem}{Theorem}[section]

\newtheorem{example}[theorem]{Example}
\newtheorem{lemma}[theorem]{Lemma}
\newtheorem{proposition}[theorem]{Proposition}
\newtheorem{remark}[theorem]{Remark}

\newenvironment{proof}[1][Proof]{\noindent \emph{#1.} }{\hfill \ 
\rule{0.5em}{0.5em}}

\makeatletter\@addtoreset{equation}{section}\makeatother
\makeatletter\@addtoreset{figure}{section}\makeatother
\makeatletter\@addtoreset{table}{section}\makeatother
\begin{document}
\title{Grid-based  lattice summation of electrostatic  potentials
by assembled rank-structured tensor approximation
}

\author{Venera Khoromskaia\thanks{Max-Planck-Institute for
        Mathematics in the Sciences, Inselstr.~22-26, D-04103 Leipzig,
        Germany ({\tt vekh@mis.mpg.de}).} \and
        Boris N. Khoromskij\thanks{Max-Planck-Institute for
        Mathematics in the Sciences, Inselstr.~22-26, D-04103 Leipzig,
        Germany ({\tt bokh@mis.mpg.de}).}
        }
 
\date{}

\maketitle

\begin{abstract}
Our recent method for low-rank tensor representation of sums of the arbitrarily 
positioned electrostatic potentials discretized on a 3D Cartesian grid 
reduces the 3D tensor summation to operations involving only 1D vectors however retaining
the linear complexity scaling in the number of potentials.
Here, we introduce and study a novel tensor approach for fast 
and accurate assembled summation of a large number of lattice-allocated potentials
represented on 3D $N\times N \times N$ grid with the computational requirements only
\emph{weakly dependent} on the number of summed potentials. 
It is based on the assembled low-rank canonical tensor representations of 
the collected potentials using pointwise sums of shifted canonical vectors 
representing the single generating function, say the Newton kernel.
For a sum of electrostatic potentials over $L\times L \times L$ lattice 
embedded in a box the required storage scales linearly in the 1D grid-size, 
$O(N )$, while the numerical cost is estimated by $O(N L)$.  
For periodic boundary conditions, the storage demand remains proportional to the 1D grid-size 
of a unit cell,   $n=N/L$,  while the numerical cost reduces to $O(N)$, that  outperforms  
the FFT-based Ewald-type summation algorithms of complexity $O(N^3 \log N)$.
The complexity in the grid parameter $N$ can be reduced even
to the logarithmic scale $O(\log N)$ by using data-sparse representation of canonical $N$-vectors
via the quantics tensor approximation.
For justification, we prove an upper bound on the quantics ranks for the 
canonical vectors in the overall lattice sum. 
 The presented approach is beneficial in applications which require further functional calculus
 with the lattice potential, say, scalar product with a function, integration or differentiation,
which can be  performed easily in tensor arithmetics on large 3D grids with 1D cost.
Numerical tests illustrate the performance of the tensor summation method
and confirm the estimated bounds on the tensor ranks.

\emph{This paper is an essentially improved version of the preprint of the Max-Planck 
Institute for Mathematics in the Sciences \cite{VeKhorEwald:13}.}
\end{abstract}



\section{Introduction}\label{sec:introduct}

There are several challenges in the numerical treatment of periodic and perturbed 
periodic systems in quantum chemical computations for crystalline, metallic and polymer-type 
compounds, see \cite{CRYSTAL:2000,CRYSCOR:12,KuScuser:04} and 
\cite{VolUsSchDePau:11,LoUsvyatSch_P1:11,CanEhrMad:2012,LSJ:2010}.
One of them is the lattice summation of electrostatic potentials 
of a large number of nuclei 
distributed on a fine 3D computational grid.  
This problem is also considered to be a demanding computational task
in the numerical treatment of long-range electrostatic interactions in molecular dynamics simulations 
of  large solvated biological systems \cite{TouBo_Ewald:96,Hune_Ewald:99,DesHolm:98}. 
In the latter applications the efficient calculation of quantities like potential energy function 
or interparticle forces remains to be of  main interest.

Tracing back to Ewald summation techniques \cite{Ewald:27},
the development of lattice-sum methods in numerical simulation of 
particle interactions in large molecular systems has led to established 
algorithms for evaluating long-range electrostatic potentials of multiparticle systems,
see for example \cite{DYP:93,PolGlo:96,TouBo_Ewald:96,Hune_Ewald:99,DesHolm:98,LinTorn2:12} 
and references therein. 
These methods usually combine the original Ewald summation approach 
with the Fast Fourier Transform (FFT) or fast multipole methods \cite{RochGreen:87}. 

The Ewald summation techniques were shown to be particularly attractive for
computation of the potential energies and forces of many-particle systems with 
long-range interaction potential in periodic boundary conditions.
They are based on the spacial separation of a sum of potentials into two parts,
the short-range part treated in the real space, and the long-range part
whose sum converges in the reciprocal space. The fast multipole method is 
used for more unstructuredly distributed  potentials, where the interactions between 
closely positioned potentials are calculated directly, and the distant
interactions are calculated by using the hierarchical clusters.

Here we propose the new approach for   calculation of lattice sums  
based on the assembled low-rank tensor-product approximation of the electrostatic potentials 
(shifted Newton kernels) discretized on large $N\times N\times N$ 3D Cartesian grid. 
A sum of potentials is represented on the 3D uniform grid in the whole computational box, 
  as a low-rank tensor with storage $O(N)$, containing tensor-products of vectors 
of a special form. This remarkable approach is  initiated by  our
former numerical observations in \cite{KhKh:06,VeKh_Diss:10} that the Tucker tensor rank 
of the 3D lattice sum of discretized Slater functions remains uniformly bounded, 
nearly independent  of the number of single Slater functions in a sum. 

As a building block we use the separable tensor-product representation (approximation) of 
a single Newton kernel $\frac{1}{r}$ in a given computational box, which provides the electrostatic
potential  at any point of an $N\times N\times N$-grid, but needs only $O(N)$ storage
due to the canonical tensor format (\ref{CP_form}). This algorithm, introduced in \cite{BeHaKh:08},
includes presentation of the potential in a form of a weighted sum of Gaussians obtained 
by the sinc-quadrature approximation to the
integral Laplace transform of the kernel function $\frac{1}{r}$, \cite{Braess:95,HaKhtens:04I}. 
Then by shifting and summation of the single canonical tensors one can construct a sum of 
electrostatic potentials located at the positions of nuclei in a molecule.
This scheme of \it{direct} tensor summation was introduced in \cite{KhorVBAndrae:11,VKH_solver:13} 
for calculation of the one-electron integrals in the black-box Hartree-Fock solver by 
grid-based tensor numerical methods\footnote{The accuracy of tensor-based calculations 
is close to accuracy of benchmark Hartree-Fock packages based on
analytical evaluation of the corresponding integrals \cite{KhorVBAndrae:11,VKH_solver:13}.}. 
It is well suited for the case of arbitrary positions of potentials, like for example 
nuclei in a molecule, however the rank of the resulting canonical tensor is approximately 
proportional to a number of summed potentials.

In this paper we introduce a novel grid-based \it{assembled} tensor summation method which matches well
for lattice-type and periodic molecular systems and yields enormous reduction in storage 
and time of calculations.
The resulting canonical tensor representing the total sum of a large number of potentials
contains the same number of canonical vectors as a tensor for a single potential. 
However, these  vectors have another content: they collect
the whole data from the 3D lattice by capturing the periodic shape
of the total 3D potential sum (represented on the grid) onto a few 1D canonican vectors,
as it is shown in Figures \ref{fig:3DPeriod_Dm8} 
and \ref{fig:3DStructCanVect} in 
sections \S  \ref{ssec:nuclear_extend} and \ref{ssec:TensorSumPeriod}.
This agglomeration is performed in a simple algebraic way by pointwise
sums of   shifted canonical vectors representing the generating function, e.g.
$\frac{1}{r}$. The presented numerical calculations confirm that the difference between the
total potentials obtained by an assembled tensor-product sum  and by a direct 
canonical sum is close to machine precision, see Figure \ref{fig:Ewald_comp}.

Thus, the adaptive global  decomposition of a sum of interacting potentials 
can be computed with a high accuracy, and in a completely algebraic way. 
The resultant potential is represented simultaneously on the fine 3D 
Cartesian grid in the whole computational box, both in the framework 
of a finite lattice-type cluster, or of a supercell in a periodic setting. 
The corresponding rank bounds for the tensor representation of the sums
of potentials are proven. Our grid-based tensor approach 
is beneficial in applications requiring further functional calculus
with the lattice potential sums, for example, interpolation, scalar product with a function, 
integration or differentiation (computation of energies or forces),
which can be  performed  on large 3D grids  using tensor arithmetics of sub-linear cost 
\cite{VeKh_Diss:10,KhKh3:08} (see Appendix).
This advantage makes the tensor method promising in electronic structure calculations, 
for example, in computation of  projections of the sum of electrostatic potentials 
onto some basis sets like  molecular or atomic Gaussian-type orbitals.

In the case of a $L\times L \times L$ lattice cluster in a box the storage size 
is shown to be bounded by $O(L)$, while the summation cost is estimated by $O(N L)$. 
The latter can be  reduced to the logarithmic scaling in the grid size, $O(L \log N )$, 
by using the quantized approximation  of long canonical vectors 
(QTT approximation method  \cite{KhQuant:09}, see Appendix). 
For a lattice cluster in a box both the fast multipole, FFT as well as the so-called $P^3 M$
methods if applicable scale at least  linear-logarithmic in the number of 
particles/nuclear charges on a lattice, $O(L^3 \log L)$, see \cite{PolGlo:96,DesHolm:98}.

For periodic boundary conditions, the respective lattice summations are reduced
to 1D operations on short canonical vectors of size $n = N/L$, 
being the restriction (projection) of the global $N$-vectors onto the unit cell.
Here $n$ denotes merely the number of grid points per unit cell.
In this case, storage and computational costs are reduced to $O(n)$ and $O(L n)$,
respectively, while the traditional FFT-based approach scales at least cubically 
in $L$, $O(L^3 \log L)$. Due to low cost of the tensor method in the limit of large 
 lattice size $L$,  the conditionally convergent sums in periodic setting can be 
regularized by subtraction of the constant term  which can be evaluated numerically 
 by the Richardson extrapolation on a sequence of lattice parameters $L, 2L, 4L$ etc. 
(see \S\ref{ssec:TensorSumPeriod}).
Hence, in the new framework the analytic treatment of the conditionally convergent sums
is no longer required.

It is worth to note that the presented tensor method is applicable to the lattice 
sums of rather general interaction potentials which allow an efficient 
local-plus-separable approximation.
In particular, along with Coulombic systems, it can be applied to a wide class of commonly 
used interaction potentials, for example, to the Slater, Yukawa, Stokeslet, Lennard-Jones or 
van der Waals interactions. 
In all these cases the existence of low-rank grid-based tensor approximation
can be proved and this approximation can be constructed numerically by analytic-algebraic methods 
as in the case of the Newton kernel.   
Our tensor approach can be extended to slightly perturbed periodic systems, for example,
to the case of few vacancies in the spacial distribution of electrostatic potentials, or
a small perturbation in positions of electron charges and other defects.
The more detailed discussion of these issues is beyond the scope of the present paper,
and is the topic of forthcoming papers.

Notice that the tensor numerical methods are  now recognized as the a powerful tool for 
solution of multidimensional partial differential equations (PDEs) discretized by 
traditional grid-based schemes. Originating from the DMRG-based matrix product states 
decomposition in quantum physics and chemistry \cite{White:93}  
and coupled with tensor multilinear algebra \cite{Kolda,KhorSurv:10,GraKresTo:13}, 
the approach was recently developed to the new branch of numerical analysis, 
tensor numerical methods, providing efficient algorithms for solving multidimensional PDEs 
with linear complexity scaling in the dimension \cite{KhorLecZuri:10}. 
One of the first steps in development of tensor numerical methods
was the 3D grid-based tensor-structured method for solution of the nonlinear Hartree-Fock 
equation \cite{Khor1:08,KhKhFl_Hart:09,VeKh_Diss:10,VKH_solver:13}
based on the efficient algorithms for the grid-based calculation of the 3D convolution 
integral operators in 1D complexity.

The reminder of the paper is structured as follows. In Section \ref{sec:NewtTens}  
we recall the   low-rank approximation to the single Newton kernel 
(electrostatic potential)
in the canonical tensor format and direct tensor calculation of the total potential 
sum of arbitrarily positioned potentials in a box. 
Section \ref{sec:BlockStrCoreHam} presents the main results of this paper 
describing the assembled low-rank tensor summation of potentials on a lattice 
in a bounded rectangular box, as well as  in the periodic setting.
The storage estimates and complexity analysis are provided.
We give numerical illustrations to the structure of assembled canonical vectors and
the results on accuracy and times of tensor summations over large 3D lattice. 
In Section \ref{sec:QTT-Sum}, we prove  the low QTT-rank approximation of the canonical vectors
in the lattice sum of the Newton kernels that justifies the logarithmic (in the grid size $N$)
storage cost of the tensor summation scheme.
Section \ref{sec:Conclusions} concludes the paper. 
For the readers convenience,
Appendix describes the basic notions of rank-structured tensor formats.

\section{Tensor decomposition of the electrostatic potential} \label{sec:NewtTens}

\subsection{Grid-based canonical representation of the Newton kernel} 
\label{ssec:CoulombUnit}

Methods of separable approximation to the 3D Newton kernel (electrostatic potential) 
using the Gaussian sums have been addressed in the chemical and mathematical literature 
since \cite{Boys:56} and \cite{Braess:BookApTh,Braess:95}, respectively.

In this section, we briefly recall the grid-based method for the low-rank tensor 
representation of the 3D Newton kernel $\frac{1}{\|x\|}$
by its projection onto the set
of piecewise constant basis functions, see \cite{Khor1:08} for more details.
Based on the results in \cite{GHK:05,HaKhtens:04I,BeHaKh:08},
this approximation can be proven to converge almost exponentially in the rank parameter. 
For the readers convenience, we now recall the main 
ingredients of this tensor approximation scheme \cite{BeHaKh:08}.


In the computational domain  $\Omega=[-b/2,b/2]^3$, 
let us introduce the uniform $n \times n \times n$ rectangular Cartesian grid $\Omega_{n}$
with the mesh size $h=b/n$.
Let $\{ \psi_\textbf{i}\}$ be a set of tensor-product piecewise constant basis functions,
$  \psi_\textbf{i}(\textbf{x})=\prod_{\ell=1}^d \psi_{i_\ell}^{(\ell)}(x_\ell)$
for the $3$-tuple index $\tb{i}=(i_1,i_2,i_3)$, $i_\ell \in I=\pl{1,...,n}$, $\ell=1,\, 2,\, 3 $.
The Newton kernel can be discretized by its projection onto the basis set $\{ \psi_\textbf{i}\}$
in the form of a third order tensor of size $n\times n \times n$, given point-wise as
\begin{eqnarray}
\mathbf{P}:=\pc{p_\tb{i}} \in \mathbb{R}^{n\times n \times n},  \quad
 p_\tb{i} = 
\int_{\mathbb{R}^3}
  \frac{\psi_{\tb{i}}(\textbf{x})}{\|\mathbf{x}\|} \,\, \mathrm{d}\tb{x}.
  \label{galten}
\end{eqnarray}
The low-rank canonical decomposition of the 3D tensor $\mathbf{P}$ is based on using exponentially convergent 
$\operatorname*{sinc}$-quadratures for approximation of the Laplace-Gauss transform to the Newton kernel,
\begin{align} \label{laplace} 
\frac{1}{\|\mathbf{x}\|}  & =  \frac{1}{\sqrt{\pi}} \int_{\mathbb{R}} 
e^{- t^2\|\mathbf{x}\|^2} \,\mathrm{d}t  = 
\frac{1}{\sqrt{\pi}} \int_{\mathbb{R}} \prod_{\ell=1}^3 e^{-t^2 (x_\ell)^2} \,\mathrm{d}t , 
\quad \|\mathbf{x}\| > 0  .
\end{align} 
 Following \cite{HaKhtens:04I,BeHaKh:08}, for any fixed $x\in \mathbb{R}^3$,
we apply the $\operatorname*{sinc}$-quadrature approximation
\begin{equation} \label{eqn:sinc_Newt}
 \frac{1}{\|\mathbf{x}\|} =  \frac{1}{\sqrt{\pi}} \int_{\mathbb{R}} 
e^{- t^2\|\mathbf{x}\|^2} \,\mathrm{d}t \approx 
\sum_{k=-M}^{M} g_k e^{- t_k^2\|\mathbf{x}\|^2} \quad \mbox{for} \quad \|\mathbf{x}\| > 0,
\end{equation}
where the quadrature points and weights are given by 
\begin{equation} \label{eqn:hM}
t_k=k \mathfrak{h}_M , \ g_k=\mathfrak{h}_M , \ \mathfrak{h}_M=C_0 \log(M)/M , \ C_0>0.
\end{equation}
Under the assumption $0< a \leq \|\mathbf{x}\| \leq A < \infty$
this quadrature provides the exponential convergence rate in $M$,
\begin{equation*} \label{sinc_conv}
\left|\frac{1}{\|\mathbf{x}\|} - \sum_{k=-M}^{M} g_k e^{- t_k^2\|\mathbf{x}\|^2} \right|  
\le \frac{C}{a}\, \displaystyle e^{-\beta \sqrt{M}},  
\quad \text{with some} \ C,\beta >0.
\end{equation*}
Combining \eqref{galten} and \eqref{eqn:sinc_Newt}, and taking into account the separability of the
Gaussian and basis functions, we obtain the separable approximation to each entry of
the tensor $\mathbf{P}$,
\begin{equation} \label{eqn:C_nD_0}
 p_\tb{i} \approx \sum_{k=-M}^{M} g_k   \int_{\mathbb{R}^3}
 \psi_\tb{i}(\textbf{x}) e^{- t_k^2\|\mathbf{x}\|^2} \mathrm{d}\tb{x}
=  \sum_{k=-M}^{M} g_k  \prod_{\ell=1}^{3}  \int_{\mathbb{R}}
\psi^{(\ell)}_{i_\ell}(x_\ell) e^{- t_k^2 x^2_\ell } \mathrm{d}x_\ell.
\end{equation}
Define the vector 
\begin{equation} \label{eqn:galten_int}
\textbf{b}^{(\ell)}(t_k)=\{b^{(\ell)}_{i_\ell}(t_k)\} \in \mathbb{R}^{n_\ell}
\quad \text{with } \quad b^{(\ell)}_{i_\ell}(t_k)= 
\int_{\mathbb{R}} \psi^{(\ell)}_{i_\ell}(x_\ell) e^{- t_k^2 x^2_\ell } \mathrm{d}x_\ell,
\end{equation}
then the $3$rd order tensor $\mathbf{P}$ can be approximated by 
the rank-$R$ ($R=2M+1$) canonical representation
\begin{equation*} \label{sinc_general}
    \mathbf{P} \approx  \mathbf{P}_R =
\sum_{k=-M}^{M} g_k \bigotimes_{\ell=1}^{3} \textbf{b}^{(\ell)}(t_k) \in \mathbb{R}^{n\times n \times n},
\quad g_k,\, t_k \in \mathbb{R},
\end{equation*}
where $R=2M+1$ and $M$ is chosen in such a way that in the max-norm
\begin{equation*} \label{error_control}
\| \mathbf{P} - \mathbf{P}_R \|  \le \varepsilon \| \mathbf{P}\|,
\end{equation*}
where $\varepsilon >0 $ is the given tolerance. 
This construction defines the rank-$R$ ($R\leq M+1$) canonical 
tensor
 \begin{equation} \label{eqn:Newt_canTens}
{\bf P}_{R}= \sum\limits_{q=1}^{R} {\bf p}^{(1)}_q \otimes {\bf p}^{(2)}_q 
\otimes {\bf p}^{(3)}_q
\in \mathbb{R}^{n\times n \times n},
\end{equation}
with canonical vectors obtained by renumbering $q=k+M+1$, 
${\bf p}^{(\ell)}_q = g_k \textbf{b}^{(\ell)}(t_k)\in \mathbb{R}^n$, $\ell=1,\, 2, \,3$. 
This tensor approximates the discretized 3D symmetric Newton kernel $\frac{1}{\|x\|}$ ($x\in \Omega$), 
centered at the origin, such that ${\bf p}^{(1)}_q={\bf p}^{(2)}_q={\bf p}^{(3)}_q$ ($q=1,...,R$). 

\begin{figure}[htbp]
\centering
\includegraphics[width=7.0cm]{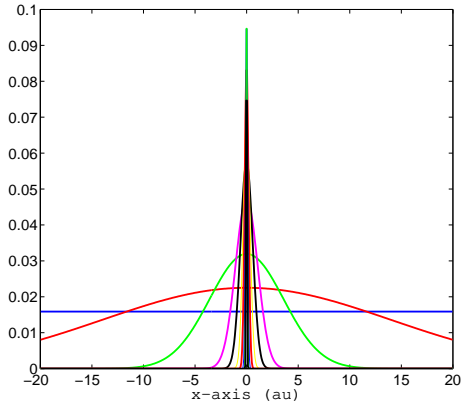}
\caption{Several vectors of the canonical tensor representation for a single 
Newton kernel along $x$-axis from a set $\{{\bf p}^{(1)}_q\}_{q=1}^R$.}
\label{fig:Newton}  
\end{figure}

Figure \ref{fig:Newton} displays several vectors of the canonical tensor representation 
for a single Newton kernel along $x$-axis from a set $\{P^{(1)}_q\}_{q=1}^R$.
Symmetry of the tensor ${\bf P}_{R}$ implies that the vectors ${\bf p}^{(2)}_q$ and ${\bf p}^{(3)}_q$ 
corresponding to $y$ and $z$-axes, respectively, are of the same shape as ${\bf p}^{(1)}_q$.
It is clearly seen that there are canonical vectors representing the long-, intermediate- and
short-range contributions to the total electrostatic potential. This interesting feature 
will be also recognized for the low-rank lattice sum of potentials 
(see Section \ref{sec:BlockStrCoreHam}).

Table \ref{Table_Times} presents CPU times for generating a canonical rank-$R$ tensor 
approximation of the Newton kernel over $n\times n\times n $ 3D Cartesian grid, 
corresponding to Matlab implementation on a terminal of the 8 AMD Opteron Dual-Core processor.
We observe a logarithmic scaling of the canonical rank $R$ in the grid size.
The compression rate denotes the ratio $n^3/(n R)$. 

\begin{table}[htb]
\begin{center}%
\begin{tabular}
[c]{|r|r|r|r|r|r|}%
\hline
grid size $n^3 $      & $8192^3$  & $16384^3$ & $32768^3$ & $65536^3$ & $131072^3$\\ 
 \hline 
 Time (sec.)         &  $6$    &    $16$  &   $61 $   &  $241$    & $1000$ \\
 \hline 
Canonical rank $R$ &  $34$ &   $37$  &   $39 $   &  $41$    & $43$ \\
 \hline 
Compression rate  & $2\cdot 10^6$ & $7\cdot 10^6$ & $2\cdot 10^7$ & $1\cdot 10^8$ & $4\cdot 10^8$  \\
 \hline 
 \end{tabular}
\caption{CPU times (Matlab) to compute canonical vectors of ${\bf P}_{R}$ for the Newton kernel 
in a box with tolerance $\varepsilon = 10^{-7}$. }
\label{Table_Times}
\end{center}
\end{table}

Notice that the low-rank canonical decomposition of the tensor ${\bf P}$ is 
the problem independent task, hence the respective canonical vectors can be precomputed at once
on very large 3D $n\times n\times n $ grid, and then stored for the multiple use. 
The storage size is bounded by $O(R n)$.

The direct tensor summation method described below
is based on the use of low-rank canonical  representation to the single Newton kernel 
${\bf P}_{R}$ in the bounding box, translated and restricted according to coordinates of the
potentials in a box.

\subsection{Direct tensor summation of electrostatic interactions} 
\label{ssec:nuclUnitC}

Here, we recall tensor summation of the electrostatic potentials 
for a non-lattice structure, that is for a (rather small) number of an arbitrarily 
distributed sources \cite{KhorVBAndrae:11,VKH_solver:13}.  
As the basic example in electronic structure calculations, we consider the nuclear potential 
operator describing the Coulombic interaction of electrons with the nuclei in a molecular 
system in a box or in a (cubic) unit cell.  
It is defined by the function $v_c(x)$ in the scaled unit cell $\Omega = [-b/2,b/2]^3$,  
\begin{equation}\label{eq Vc}
 v_c(x)=  \sum_{\nu=1}^{M_0}\frac{Z_\nu}{\|{x} -a_\nu \|},\quad
Z_\nu >0, \;\; x,a_\nu \in \Omega \subset \mathbb{R}^3,
\end{equation}
where $M_0$ is the number of nuclei in $\Omega$, and $a_\nu$,  $Z_\nu$, represent their 
coordinates and charges, respectively.

We begin with approximating the non-shifted 3D Newton kernel $\frac{1}{\|x\|}$ on the auxiliary 
extended box $\widetilde{\Omega}=[-b,b]^3 $, by its projection onto the basis set
$\{ \psi_\textbf{i}\}$ of piecewise constant functions
defined on the uniform $2n\times 2n \times 2n$ tensor grid $\Omega_{2n}$ with
the mesh size $h$, as described in Section \ref{ssec:CoulombUnit}.
This defines  the ''master`` rank-$R$ canonical tensor as above 
\begin{equation} \label{master_pot}
\widetilde{\bf P}_R= 
\sum\limits_{q=1}^{R} {\bf p}^{(1)}_q \otimes {\bf p}^{(2)}_q \otimes {\bf p}^{(3)}_q
\in \mathbb{R}^{2n\times 2n \times 2n}.
\end{equation}

For ease of exposition, we assume  that each nuclei coordinate $a_\nu$ is located
exactly at a grid-point 
$a_\nu=(i_\nu h-b/2,j_\nu h-b/2,k_\nu h-b/2)$, with some $1 \leq i_\nu,j_\nu,k_\nu \leq n $. 
Now we are able to introduce the rank-$1$ windowing operator  
${\cal W}_{\nu}={\cal W}_{\nu}^{(1)}\otimes {\cal W}_{\nu}^{(2)}\otimes {\cal W}_{\nu}^{(3)}$ for 
$\nu=1,...,M_0$ by 
\begin{equation} \label{sub_tens}
 {\cal W}_{\nu} \widetilde{\bf P}_R :=
\widetilde{\bf P}_R(i_\nu +n/2:i_\nu +3/2n;j_\nu +n/2 :j_\nu +3/2 n;k_\nu +n/2:k_\nu +3/2n)
\in \mathbb{R}^{n\times n \times n}, 
\end{equation}

With this notation, the total electrostatic potentials $v_c(x)$ in the computational 
box $\Omega$ is approximately represented by a direct canonical tensor sum
\begin{equation} \label{core_tens}
\begin{split}
 {\bf P}_{c} & = \sum_{\nu=1}^{M_0} Z_\nu {\cal W}_{\nu} \widetilde{\bf P}_R \\
             & =\sum_{\nu=1}^{M_0} Z_\nu  
\sum\limits_{q=1}^{R} {\cal W}_{\nu}^{(1)} {\bf p}^{(1)}_q \otimes 
{\cal W}_{\nu}^{(2)} {\bf p}^{(2)}_q 
\otimes {\cal W}_{\nu}^{(3)} {\bf p}^{(3)}_q\in \mathbb{R}^{n\times n \times n},
\end{split}
\end{equation}
with the rank bound 
$$
rank({\bf P}_{c})\leq M_0 R,
$$
where every rank-$R$ canonical tensor 
${\cal W}_{\nu} \widetilde{\bf P}_R \in \mathbb{R}^{n\times n \times n}$
is thought as a sub-tensor of the master tensor
$\widetilde{\bf P}_R \in \mathbb{R}^{2n\times 2n \times 2n}$ obtained 
by its shifting and restriction (windowing) onto the $n \times n \times n$ grid 
in the box $\Omega$, $\Omega_{n} \subset \Omega_{2n}$. Here a shift
from the origin is specified according to the coordinates of the corresponding nuclei, $a_\nu$,
counted in the $h$-units.

For example, the electrostatic potential centered at the origin, i.e. with $a_\nu=0$, corresponds
to the restriction of $\widetilde{\bf P}_R$ onto the initial computational box $\Omega_{n}$,
i.e. to the index set (assume that $n$ is even)
$$
\{ [n/2+i]\times[n/2+j]\times[n/2+k]\},\quad i,j,k\in \{1,...,n\}.
$$
\begin{remark}\label{rem:RankNuclPot}
The rank estimate (\ref{core_tens}) for the sum of arbitrarily positioned 
electrostatic potentials in a box (unit cell), 
$R_{c}=rank({\bf P}_{c})\leq M_0 R $, is usually too pessimistic.
Our numerical tests for moderate size molecules indicate that  the rank of the 
$(M_0 R)$-term canonical sum in (\ref{core_tens}) can be considerably reduced. 
This rank optimization can be implemented by the multigrid version of the 
canonical rank reduction algorithm, canonical-Tucker-canonical \cite{KhKh3:08}.
The resultant canonical tensor will be denoted by $\widehat{\bf P}_{c}$.
\end{remark}


The proposed grid-based representation of the exact sum of electrostatic
potentials $v_c(x)$ in a form of a tensor in a canonical format enables
its easy projection to some separable basis set,
like GTO-type atomic orbital basis often used in quantum chemical computations.
The following example illustrates calculation of the nuclear potential 
operator matrix in tensor format for molecules \cite{KhorVBAndrae:11,VKH_solver:13}.  
We show that the projection of a sum of electrostatic potentials of atoms
onto a given set of basis functions is reduced to a combination of 1D Hadamard and 
scalar products (cf. (\ref{scal})-(\ref{had}) in Appendix).   

\begin{example}\label{exm:GTOproj} \cite{KhorVBAndrae:11,VKH_solver:13} Let us 
discuss the tensor-structured calculation of the nuclear potential operator in a molecule. 
Given the set of continuous basis functions, 
\begin{equation} \label{basis}
\{g_\mu(x)\}, \quad  \mu =1,...,N_b,
\end{equation}
then each of them can be discretized by a third order tensor,
\[
{\bf G}_\mu = \left[g_\mu(x_1(i),x_2(j),x_3(k))\right]_{i,j,k=1}^n 
\in \mathbb{R}^{n\times n \times n},
\]
obtained by sampling of $g_\mu(x)$ at the midpoints  
$(x_1(i),x_2(j),x_3(k))$ of the grid-cells indexed by $(i,j,k)$.
Suppose, for simplicity, that it is a rank-$1$ canonical tensor, $rank({\bf G}_\mu)=1$, 
i.e.  
$$
{\bf G}_\mu = G_\mu^{(1)}\otimes G_\mu^{(2)} \otimes G_\mu^{(3)}
\in \mathbb{R}^{n \times n \times n},
$$
with the canonical vectors $G_\mu^{(\ell)}\in\mathbb{R}^{n} $, 
associated with mode $\ell=1,2,3$.

A sum of potentials in a box, $v_c(x)$ (\ref{eq Vc}), is represented 
in the given basis set (\ref{basis}) by a matrix $V_c=\{{v}_{km}\}\in \mathbb{R}^{N_b\times N_b}$, 
whose entries are calculated (approximated) by the simple tensor operation, see 
\cite{KhorVBAndrae:11,VKH_solver:13},
\begin{equation} \label{nuc_pot}
 {v}_{km}=  \int_{\mathbb{R}^3} v_c(x) {g}_k(x) {g}_m(x) dx \approx 
 \langle {\bf G}_k \odot {\bf G}_m ,   {\bf P}_{c}\rangle, 
\quad 1\leq k, m \leq N_b.
\end{equation}
Here ${\bf P}_{c}$ is a sum of shifted/windowed canonical tensors (\ref{core_tens})
representing the total electrostatic potential of atoms in a molecule, and 
\[
 {\bf G}_k \odot {\bf G}_m =
 (  G_k^{(1)} \odot  G_m^{(1)}  ) \otimes (G_k^{(2)} \odot  G_m^{(2)} ) \otimes
 (  G_k^{(3)} \odot  G_m^{(3)}  ) 
\]
denotes the Hadamard (entrywise) product of tensors representing the basis functions (\ref{had}), 
which is reduced to 1D products. The scalar product $\langle \cdot,\cdot \rangle$ in 
(\ref{nuc_pot}) is also reduced to 1D scalar products (\ref{scal})
due to separation of variables.
\end{example}

To conclude this section, we notice that the approximation error $\varepsilon >0$ 
caused by a separable representation of the nuclear potential 
is controlled by the rank parameter $R_{c}= rank({\bf P}_{c})\approx C\, R $,
where $C$ depends on the number of nuclei $M_0$. 
Now letting $rank({\bf G}_m) = 1$ implies that each matrix element is to be computed with 
linear complexity in $n$, $O(R \, n)$. 
The exponential convergence of the canonical approximation in the rank parameter $R$ allows us 
the optimal choice $R=O(|\log \varepsilon |)$ adjusting the overall complexity 
bound $O(|\log \varepsilon | \, n)$, independent on $M_0$.

\section{Lattice potential sums by assembled canonical  tensors} \label{sec:BlockStrCoreHam}

In the numerical treatment of extended systems the 3D summation over $L^3$ cells 
in the limit of large $L$  is considered as a hard computational task both in a bounded and in a periodic setting. 
The commonly used methods, known in the literature 
as the Ewald summation algorithms \cite{Ewald:27}, are based on a certain specific 
local-global decomposition of the Newton kernel 
(see \cite{TouBo_Ewald:96,Hune_Ewald:99,DesHolm:98})
\[
 \frac{1}{r}=\frac{\tau(r)}{r} + \frac{1-\tau(r)}{r},
\]
where the traditional choice of the cutoff function $\tau$ is the complementary error function
\[
 \tau(r)=\operatorname{erfc}(r):=\frac{2}{\sqrt{\pi}}\int_{r}^\infty \exp(-t^2) dt.
\]
In this paper, we introduce a new grid-based approach to the problem of lattice summation 
of electrostatic potentials in a box/supercell based on the fast agglomerated tensor sums.
The resulting total potential on the lattice is represented in a form of a low-rank canonical tensor. 
The main ingredients are the separable canonical tensor 
approximation of the Newton kernel and fast rank-structured tensor arithmetics.
 The proposed tensor method is not limited to a special case of   Newton 
kernel $\frac{1}{\|x\|}$, and can be applied to a general class of 
shift-invariant well separable generating potentials.

\subsection{Assembled lattice sums in a box} 
\label{ssec:nuclear_extend}

In this section, we present the efficient scheme for fast agglomerated tensor summation 
of electrostatic potentials for  a lattice in a box. 
Given the potential sum $v_c$ in the unit reference cell $\Omega=[-b/2,b/2]^d$, $d=3$,  of size 
$b\times b \times b$,  we consider an interaction potential in a bounded box
$$
\Omega_L =B_1\times B_2 \times B_3,
$$ 
consisting of a union of  $L_1 \times L_2 \times L_3$ unit cells $\Omega_{\bf k}$,
obtained  by a shift of $\Omega$ that is a multiple of
$ b$ in each variable, and specified by the lattice vector $b {\bf k}$,
${\bf k}=(k_1,k_2,k_3)\in \mathbb{Z}^d$, $0  \leq k_\ell\leq L_\ell-1 $
for $L_\ell \in \mathbb{N}$, ($\ell=1,2,3$).
Here 
$B_\ell = [-b/2 ,b/2 + (L_\ell-1) b]$ , such that the case $L_\ell=1$ corresponds to
one-layer systems in the variable $x_\ell$. Recall that by the construction $b=n h$, 
where $h >0$ is the mesh-size (same for all spacial variables).

In the case of an extended system in a box the summation 
problem for the total potential $v_{c_L}(x)$ is formulated in the 
rectangular volume $\Omega_L= \bigcup_{k_1,k_2,k_3=1}^L \Omega_{\bf k}$,
where for ease of exposition we  consider a lattice of equal sizes $ L_1= L_2=L_3=L$. 
In general, the volume box for calculations is larger than $\Omega_L$, 
by a distance of several $\Omega$ (see Figures \ref{fig:3DPeriodStruct} 
and \ref{fig:3DPeriod_Dm8}).  
On each $\Omega_{\bf k}\subset \Omega_L$, the  potential sum of interest, 
$v_{\bf k}(x)=(v_{c_L})_{|\Omega_{\bf k}}$, 
is obtained by summation over all unit cells in $\Omega_L$,
\begin{equation}\label{eqn:EwaldSumE}
v_{\bf k}(x)=  \sum_{\nu=1}^{M_0} \sum\limits_{k_1,k_2,k_3=0}^{L-1} 
\frac{Z_\nu}{\|{x} -a_\nu (k_1,k_2,k_3)\|}, \quad x\in \Omega_{\bf k}, 
\end{equation}
where $a_\nu (k_1,k_2,k_3)=a_\nu  + b {\bf k}$.
This calculation is performed at each of $L^3$ elementary cells 
$\Omega_{\bf k}\subset \Omega_L$, which usually presupposes substantial 
numerical costs for large $L$. In the presented approach these costs are essentially
reduced, as it is described further.

Figure \ref{fig:3DPeriodStruct} shows the example of a computational box with a 3D lattice-type
molecular structure of $4\times 4\times 2$ atoms and the calculated lattice sum of
electrostatic  potentials. 
\begin{figure}[tbh]
\centering
\includegraphics[width=4.5cm]{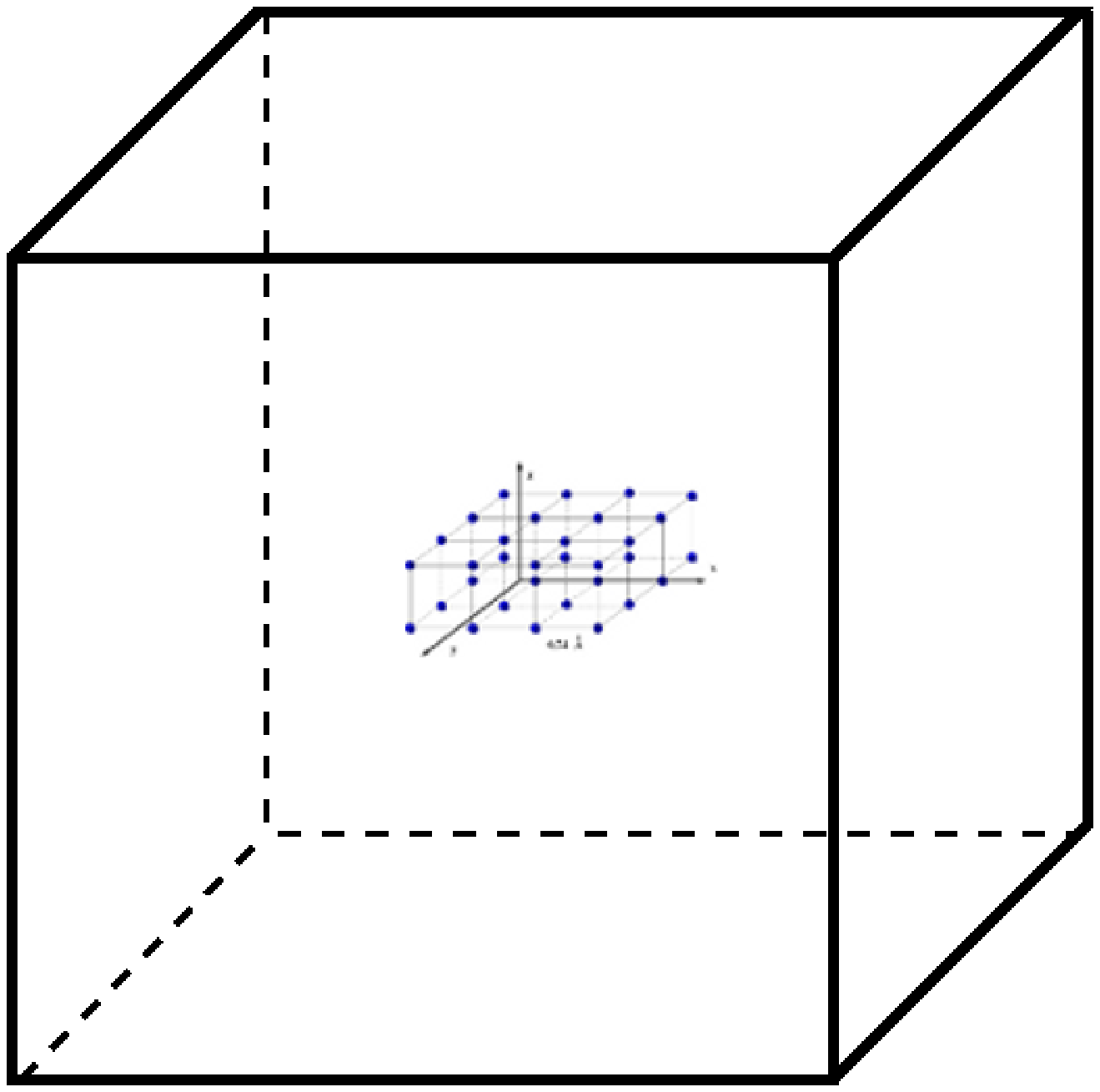} \quad \quad  
\includegraphics[width=7.0cm]{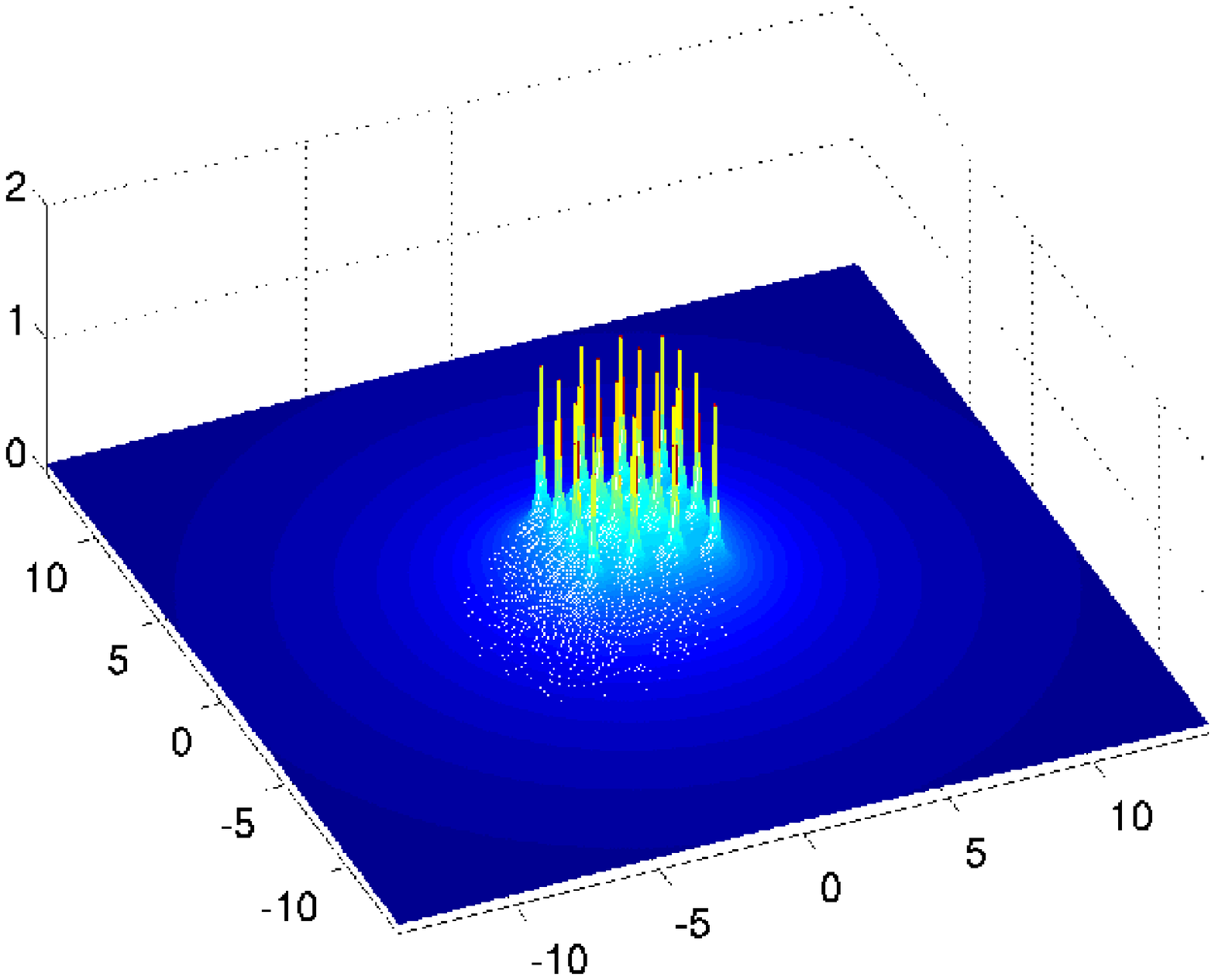}
\caption{\small  Example of $4\times 4 \times 2$ lattice compound in a computational box and
calculated potential sum.}
\label{fig:3DPeriodStruct}  
\end{figure}

Let $\Omega_{N_L}$ be the $N_L\times N_L\times N_L$ uniform grid on $\Omega_L$ with the 
same mesh-size $h$
as above, and introduce the corresponding space of piecewise constant basis functions 
of the dimension $N_L^3$. In this construction we have 
\begin{equation}\label{eqn:N0grid}
N_L= n+ n (L-1)=Ln.
\end{equation}
Similar to (\ref{master_pot}), we employ  the rank-$R$ ''master`` tensor defined 
on the auxiliary box $\widetilde{\Omega}_{L}$ by scaling ${\Omega}_{L}$ with factor $2$.
$$
\widetilde{\bf P}_{{L},R}= \sum\limits_{q=1}^{R} {\bf p}^{(1)}_q 
\otimes {\bf p}^{(2)}_q \otimes {\bf p}^{(3)}_q 
\in \mathbb{R}^{2 N_L\times 2N_L\times 2N_L},
$$
and let ${\cal W}_{\nu(k_i)}$, $i=1,2,3$, be the directional windowing operators 
associated with the lattice vector ${\bf k}$. In the next theorem we prove
the storage and numerical costs for the lattice sum  of single potentials,
each represented by a canonical rank-$R$ tensor, which corresponds to the choice
 $M_0=1$, and $a_1 =0$ in (\ref{eqn:EwaldSumE}). 
In this case the windowing operator ${\cal W} = {\cal W}_{({\bf k})}$ 
specifies the shift by the lattice vector $b{\bf k}$.

\begin{theorem}\label{thm:sumCaseE}
The projected tensor of the interaction potential 
$v_{c_L} (x)$, $x\in \Omega_{L}$,  
representing the full lattice sum over $L^3$ charges can be presented by the 
canonical tensor ${\bf P}_{c_L}$ with the rank $R$,
\begin{equation}\label{eqn:EwaldTensorGl}
{\bf P}_{c_L}= 
\sum\limits_{q=1}^{R}
(\sum\limits_{k_1=0}^{L-1}{\cal W}_{({k_1})} {\bf p}^{(1)}_{q}) \otimes 
(\sum\limits_{k_2=0}^{L-1} {\cal W}_{({k_2})} {\bf p}^{(2)}_{q}) \otimes 
(\sum\limits_{k_3=0}^{L-1}{\cal W}_{({k_3})} {\bf p}^{(3)}_{q}).
\end{equation}
The numerical cost and storage size are estimated by $O(R L N_L)$ and $O(R N_L)$,
respectively, where $N_L$ is the univariate grid size as in (\ref{eqn:N0grid}).
\end{theorem}
\begin{proof}
For the moment, we fix index $\nu=1$ in (\ref{eqn:EwaldSumE}) 
and consider only the second sum  defined on the 
complete domain $\Omega_L$,
\begin{equation}\label{eqn:EwaldSumGl}
{v}_{c_L} (x)=   \sum\limits_{k_1,k_2,k_3=0}^{L-1}
\frac{Z }{\|{x} - b {\bf k} \|}, \quad x\in  \Omega_L.
\end{equation}
Then the projected tensor representation of ${v}_{c_L} (x)$ takes the form (omitting factor $Z $)
\[
  {\bf P}_{c_L}= \sum\limits_{k_1,k_2,k_3=0}^{L-1}  {\cal W}_{\nu({\bf k})}  
 {\bf P}_{{L},R}= \sum\limits_{k_1,k_2,k_3=0}^{L-1} \sum\limits_{q=1}^{R}
{\cal W}_{ ({\bf k})}( {\bf p}^{(1)}_{q} \otimes {\bf p}^{(2)}_{q} \otimes {\bf p}^{(3)}_{q})
\in \mathbb{R}^{N_L\times N_L  \times N_L},
\]
where the 3D shift vector is defined by  ${\bf k}\in \mathbb{Z}^{L\times L\times L}$. 
Taking into account the separable representation 
of the $\Omega_L$-windowing operator (tracing onto $N_L\times N_L\times N_L$ window),
$$
{\cal W}_{({\bf k})}={\cal W}_{ (k_1)}^{(1)}\otimes {\cal W}_{ (k_2)}^{(2)}
\otimes {\cal W}_{ (k_3)}^{(3)},
$$ 
we reduce the above summation to 
\begin{equation}\label{eqn:fullsum}
  {\bf P}_{c_L}=  \sum\limits_{q=1}^{R}\sum\limits_{k_1,k_2,k_3=0}^{L-1}
{\cal W}_{ ({k_1})} {\bf p}^{(1)}_{q} \otimes {\cal W}_{ ({k_2})} {\bf p}^{(2)}_{q} 
\otimes {\cal W}_{ ({k_3})} {\bf p}^{(3)}_{q}.
\end{equation} 
To reduce the large sum over the full 3D lattice, we use the following property of a sum
of canonical tensors with equal ranks $R$ and with two coinciding factor matrices: 
the concatenation (\ref{eqn:conc})
in the third mode $\ell$  can be reduced to point-wise summation of respective canonical vectors,
\begin{equation}
\label{eqn:conc2sum}
C^{(\ell)} =[{\bf a}_1^{(\ell)}+ {\bf b}_1^{(\ell)}, \ldots ,
{\bf a}_{R_a}^{(\ell)}+ {\bf b}_{R_b}^{(\ell)}],
\end{equation} 
thus preserving the same rank parameter $R$ for the resulting sum.
Notice that for each fixed $q$ the inner sum in (\ref{eqn:fullsum}) satisfies the above
property. Repeatedly using this property to a large number of canonical tensors, 
the  3D-sum (\ref{eqn:fullsum}) can be simplified to a rank-$R$ tensor obtained 
by 1D summations only,
\[
\begin{split}
 {\bf P}_{c_L} & = \sum\limits_{q=1}^{R}
(\sum\limits_{k_1=0}^{L-1}{\cal W}_{ ({k_1})} {\bf p}^{(1)}_{q}) \otimes
 (\sum\limits_{k_2,k_3=0}^{L-1} {\cal W}_{ ({k_2})} {\bf p}^{(2)}_{q} 
\otimes {\cal W}_{ ({k_3})} {\bf p}^{(3)}_{q})\\
 &= \sum\limits_{q=1}^{R} 
 (\sum\limits_{k_1=0}^{L-1}{\cal W}_{ ({k_1})} {\bf p}^{(1)}_{q}) \otimes 
(\sum\limits_{k_2=0}^{L-1} {\cal W}_{ ({k_2})} {\bf p}^{(2)}_{q}) \otimes 
(\sum\limits_{k_3=0}^{L-1}{\cal W}_{ ({k_3})} {\bf p}^{(3)}_{q}).
\end{split}
\]
The numerical cost can be estimated by taking into account the 
standard properties of canonical tensors.
\end{proof}
 
\begin{figure}[tbh]
\centering
\includegraphics[width=7.0cm]{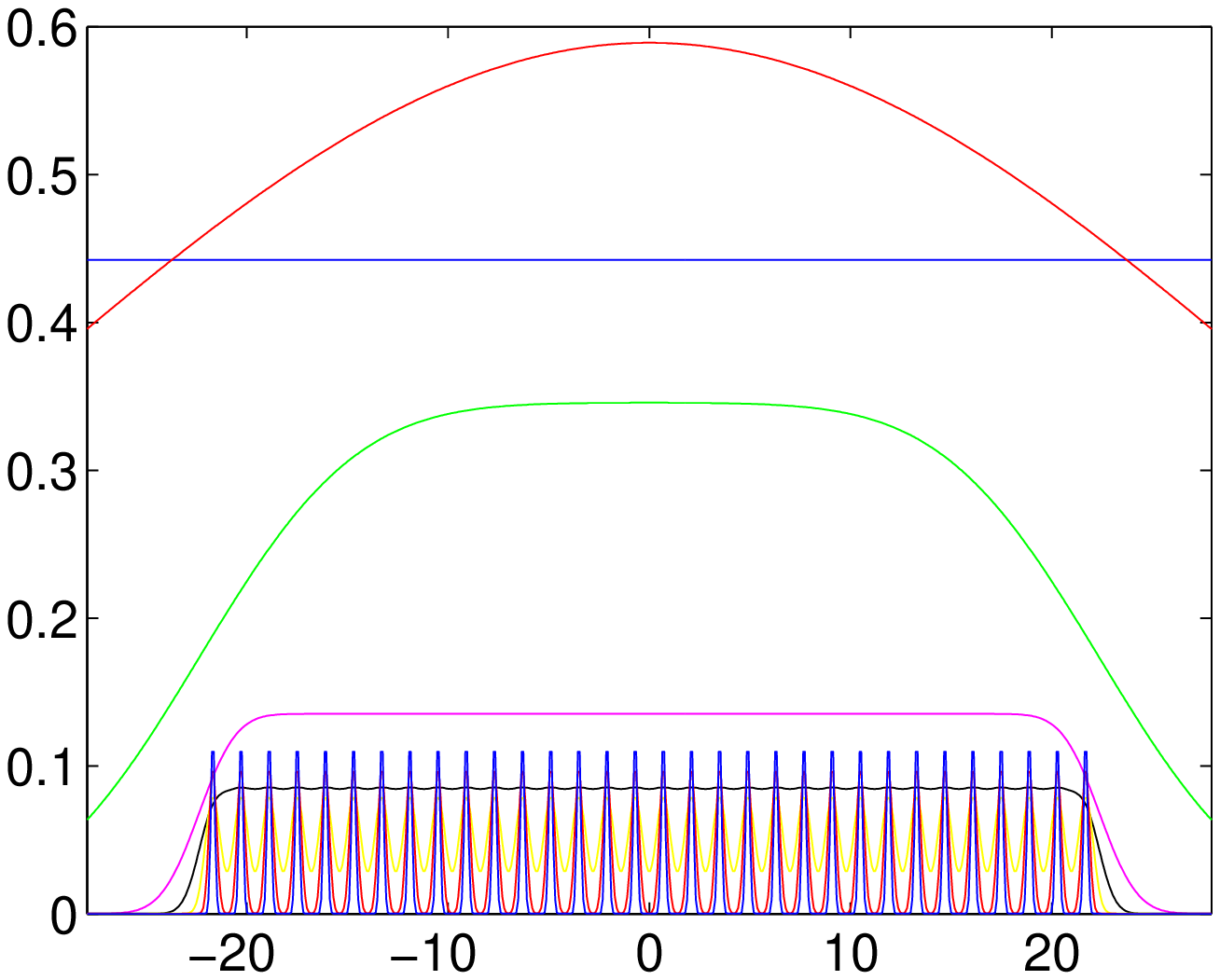}
\includegraphics[width=7.0cm]{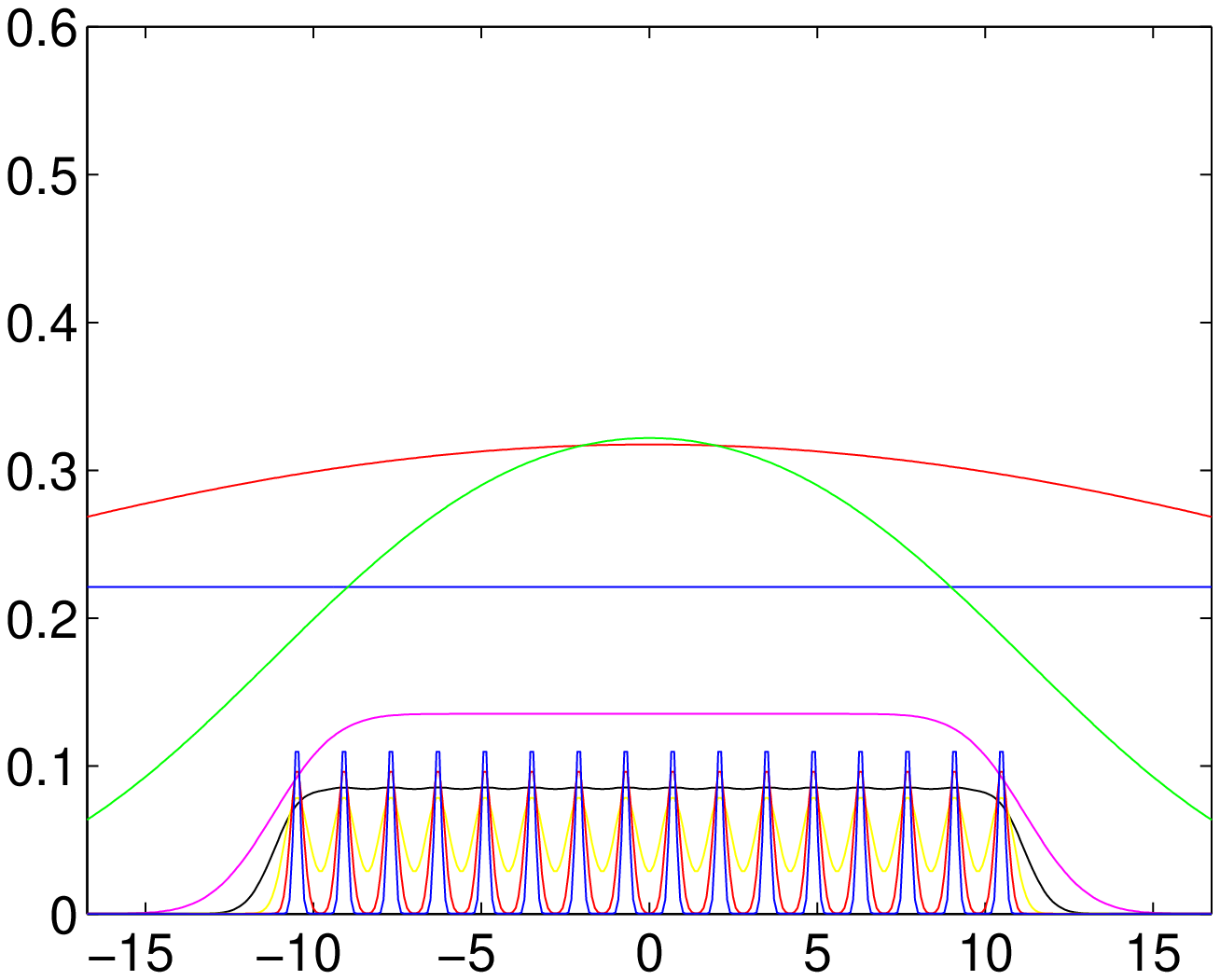}
\includegraphics[width=7.0cm]{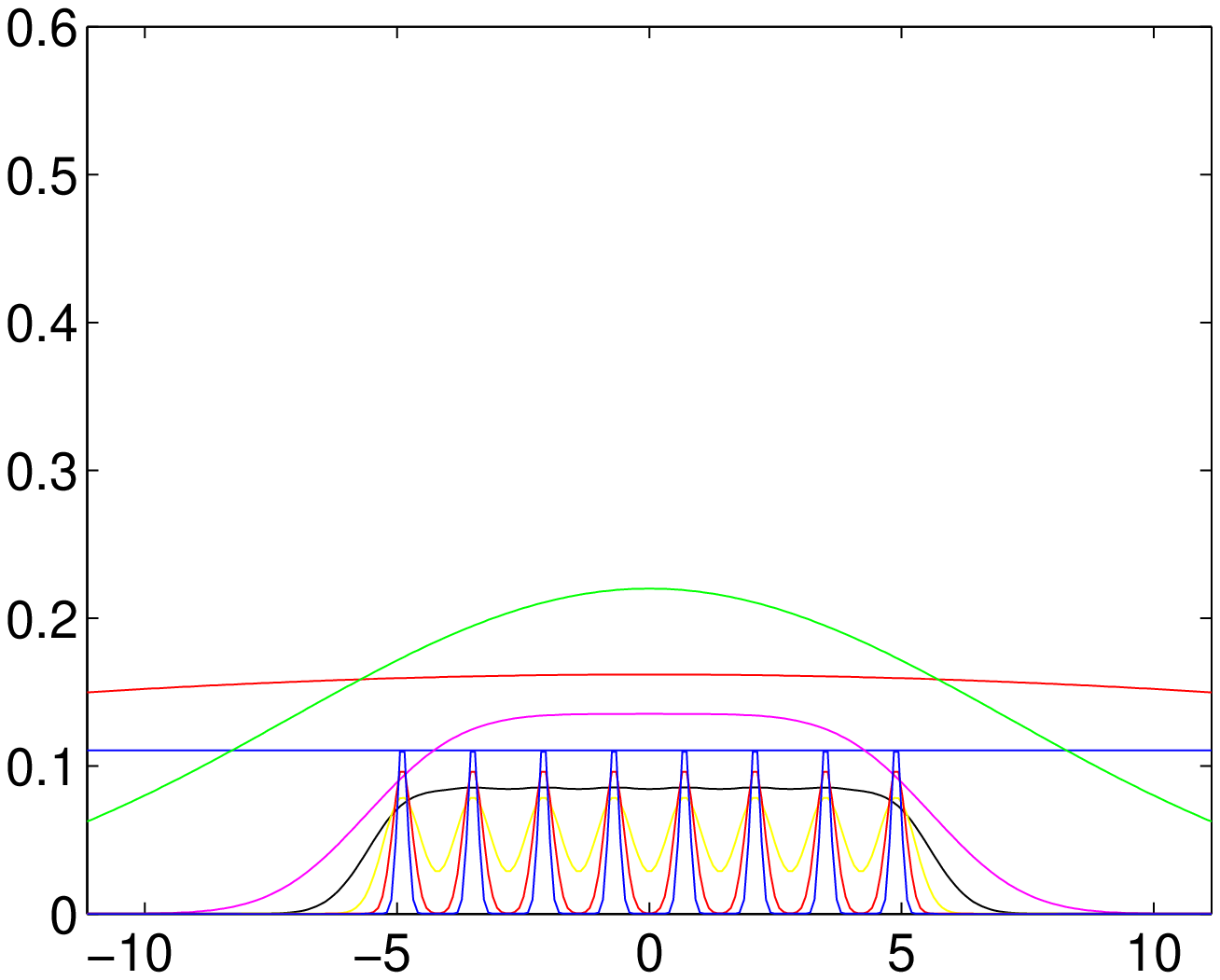}
\includegraphics[width=7.0cm]{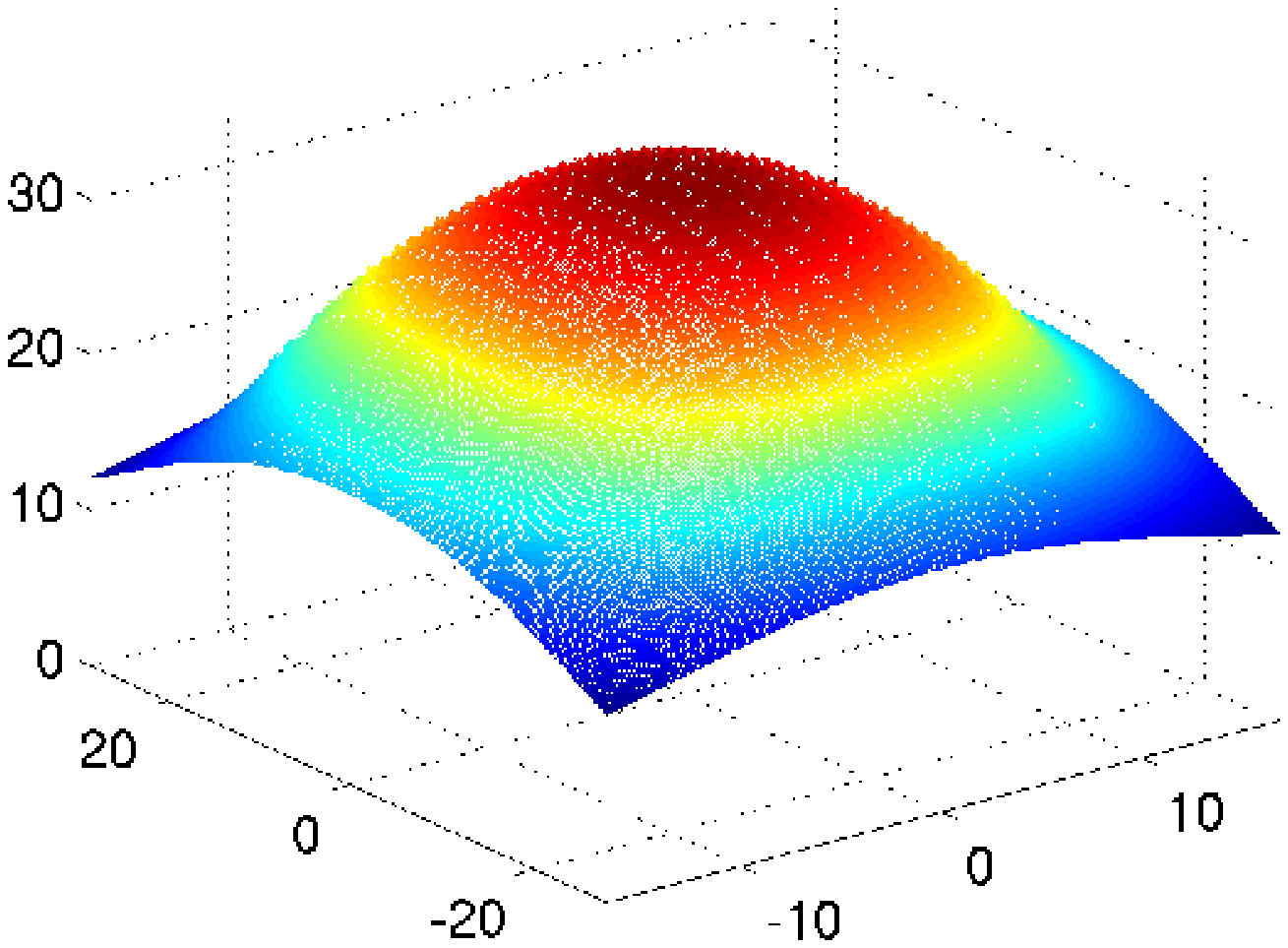}
\caption{Agglomerated canonical vectors for a sum of electrostatic potentials for a cluster of
$32\times 16\times 8$ Hydrogen atoms in a rectangular box of size $\sim 55.4 \times 33.6 \times 22.4$ $au^3$. 
Top left-right: vectors along $x$- and $y$-axes, respectively;
bottom left: vectors along $z$-axis. Right bottom: the resulting sum of $4096$ nuclei potentials at 
cross-section with $z=0.83$ au.}
\label{fig:3DPeriod_Dm8}  
\end{figure}

\begin{remark}
 For the general case $M_0 >1$, the weighted summation over $M_0$ charges leads 
to the low-rank tensor representation,
\begin{equation}\label{eqn:EwaldTensorM0}
{\bf P}_{c_L}= \sum\limits_{\nu=1}^{M_0} Z_\nu \sum\limits_{q=1}^{R}
(\sum\limits_{k_1=0}^{L-1}{\cal W}_{\nu({k_1})} {\bf p}^{(1)}_{q}) \otimes 
(\sum\limits_{k_2=0}^{L-1} {\cal W}_{\nu({k_2})}{\bf p}^{(2)}_{q}) \otimes 
(\sum\limits_{k_3=0}^{L-1}{\cal W}_{\nu({k_3})} {\bf p}^{(3)}_{q}).
\end{equation}
\end{remark}

\begin{remark}
The previous construction applies to the uniformly spaced positions of charges.
However, our tensor summation method remains valid for a
non-equidistant $L \times L \times L $  tensor lattice.
\end{remark}

Figure  \ref{fig:3DPeriod_Dm8} illustrates the shape of canonical vectors for
the $32\times 16\times 8$ lattice sum in a box (summation of $4096$ potentials). 
Here the canonical rank $R = 25$, and $\varepsilon=10^{-6}$.
It demonstrates how the assembled vectors composing the tensor lattice sum 
incorporate simultaneously 
the canonical vectors of shifted Newton kernels.
It can be seen that separate canonical vectors capture the 
local, intermediate and long-range contributions to the total sum.
\begin{figure}[htb]
\centering
\includegraphics[width=6.0cm]{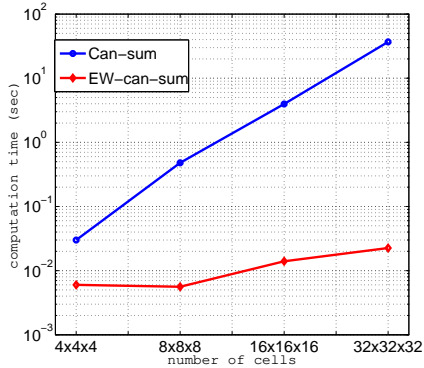}
\caption{ CPU times ($\log$ scaling) for calculating the sum of Coulomb  potentials  
over 3D  $L \times L\times L$ lattice by using direct canonical tensor summation 
(blue line) and assembled lattice summation (red line).}
\label{fig:Scalin_in_L_Ewald}  
\end{figure}

\begin{table}[tbh]
\begin{center}%
\begin{tabular}
[c]{|r|r|r|r|r|r|r|r|}%
\hline
$L$                 & $2$    & $4$     & $8$     & $16$     & $32$  & $64$ & $128$ \\
\hline
Times for $L\times L\times 1$ & $0.003$ & $0.004$ & $0.0073$ & $0.025$  & $0.128$  & $0.65$ & $2.96$ \\
\hline \hline
Times for $L\times L\times L$ & $0.003$ & $0.005$ & $0.0098$ & $0.039$  & $0.19$  & $0.88$ & $4.01$ \\
\hline  
$L^3$         & $8$ & $64$ & $512$ & $4096$  & $32768$  & $262144$ & $2097152$ \\
\hline \hline
 \end{tabular}
\caption{Times (sec) vs. $L$ for the assembled calculation of the lattice potential
${\bf P}_{c_L}$ for the clusters $L \times L \times 1$ and $L \times L \times L $. 
The last line shows the total number of cells (here Hydrogen atoms) for 
the clusters of type $L\times L\times L$.}
\label{Table_timesL}
\end{center}
\end{table}
\begin{figure}[tbh]
\centering
\includegraphics[width=7.5cm]{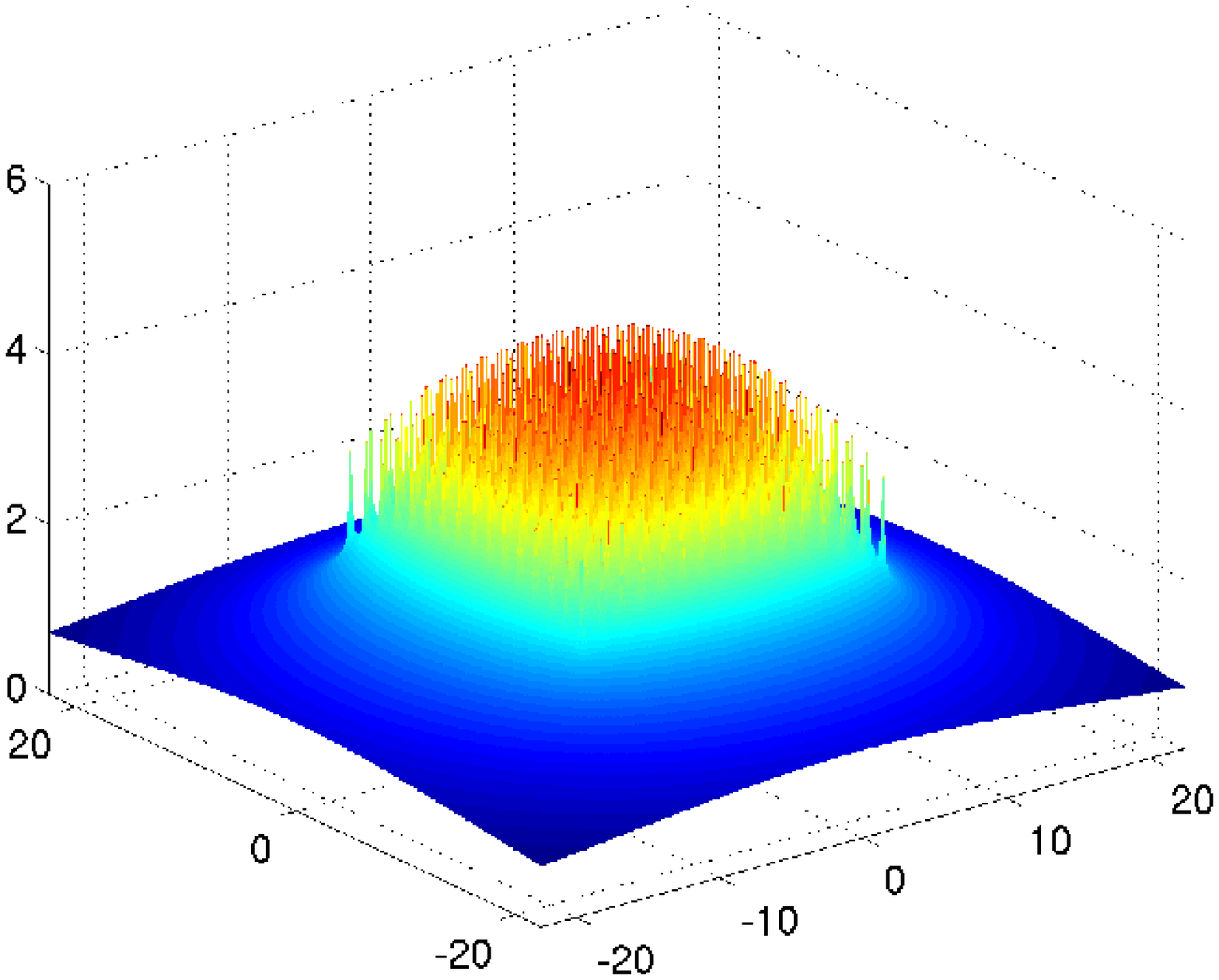}\quad \quad
\includegraphics[width=7.5cm]{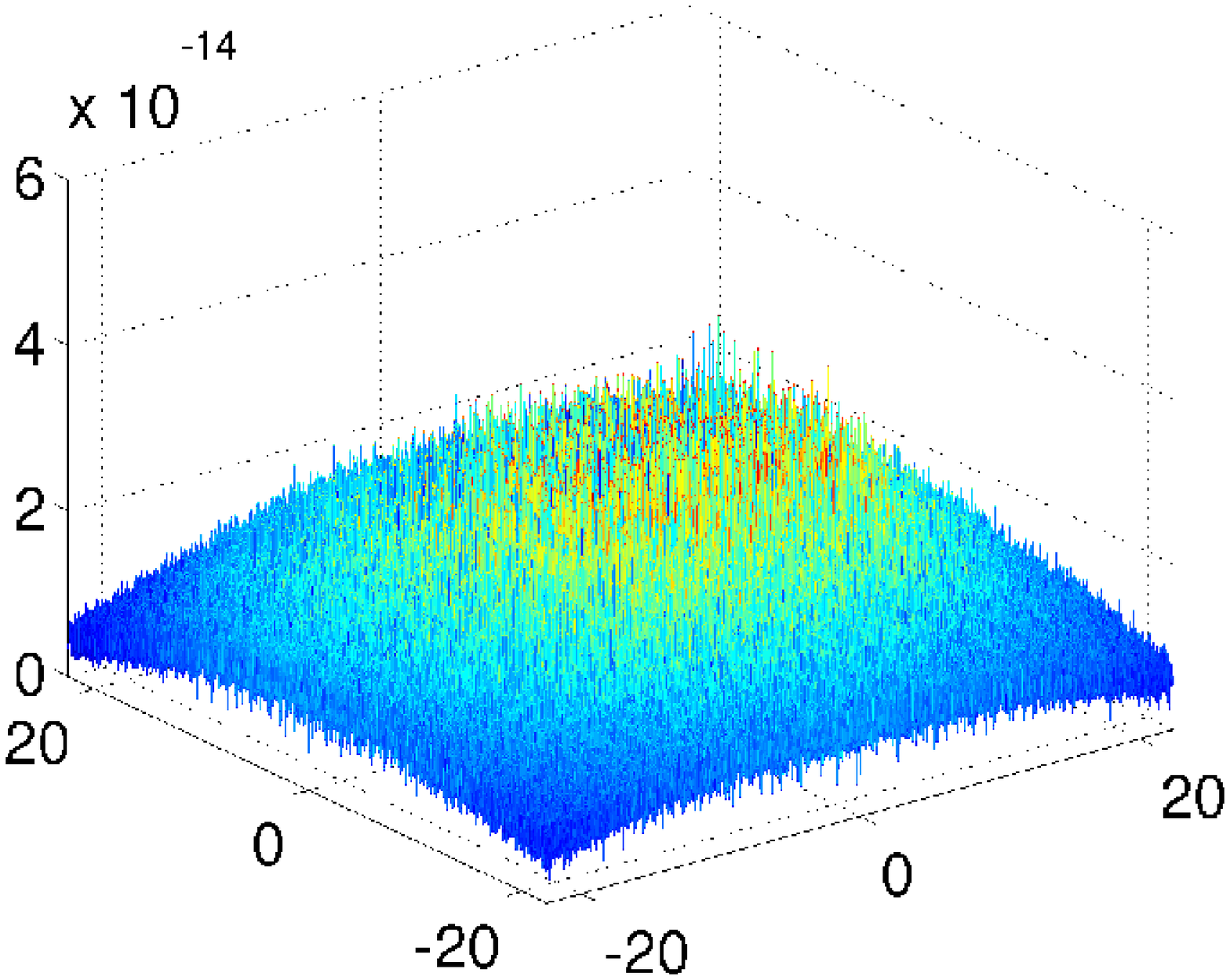}
\caption{Left: The electrostatic potential of the cluster of $16 \times 16 \times 2$ Hydrogen atoms
in a box (512 atoms). Right: the absolute error of the assembled tensor sum on this cluster
by (\ref{eqn:EwaldTensorGl}) with respect to the direct tensor summation (\ref{core_tens}).}
\label{fig:Ewald_comp}  
\end{figure}

The canonical tensor representation (\ref{eqn:EwaldTensorGl}) reduces 
dramatically the numerical costs and storage consumptions. Figure \ref{fig:Scalin_in_L_Ewald}
compares the direct and assembled tensor summation methods 
(grid-size of a unit cell, $n=256$).
Contrary to the direct canonical summation of the nuclear potentials on a 3D lattice that 
scales at least linearly in  the size of the cubic lattice, $N_L L^3$ (blue line), 
the CPU time for directionally agglomerated  canonical summation 
in a box via (\ref{eqn:EwaldTensorGl}) scales as  $N_L L$ (red line). 

Table \ref{Table_timesL} illustrates  complexity scaling $O(N_L L)$ for  
tensor lattice summation in a box of size $L\times L\times 1$ and $L\times L\times L$.
We observe the $L^2$ scaling which confirms our theoretical estimates.

Figure \ref{fig:Ewald_comp} compares  the tensor sum obtained by the assembled 
canonical vectors with the results of direct tensor sum for same configuration.
The absolute difference of the corresponding sums for a cluster of $16 \times 16\times 2$ cells 
(here a cluster of 512 Hydrogen atoms) is close to machine accuracy $\sim 10^{-14}$.

\subsection{Assembled tensor sums in a periodic setting}
\label{ssec:TensorSumPeriod}

In the periodic case we introduce the periodic cell 
${\cal R}= b \mathbb{Z}^d$, $d=1,2,3$, and consider a 3D $T$-periodic 
supercell of size $T\times T\times T$, with $T= bL$. 
The total electrostatic potential in $\Omega_L$ is obtained by the respective summation 
over the supercell $\Omega_L$ for possibly large $L$.
Then  the 
electrostatic potential in any of $T$-periods is obtained by replication of the 
respective data from $\Omega_L$.

The potential sum $v_{c_L}(x)$ is designated at each elementary 
unit-cell in $\Omega_L$ by the same value (${\bf k}$-translation invariant).
Consider the case $d=3$. Supposing for simplicity 
that $L$ is odd, $L=2p+1$, the reference value of $v_{c_L}(x)$ will be computed at the central cell
$\Omega_0$, indexed by $(p+1,p+1,p+1)$, by summation over all the contributions from $L^3$ 
elementary sub-cells in $\Omega_L$,
\begin{equation}\label{eqn:EwaldSumP}
 v_0(x)=  \sum_{\nu=1}^{M_0} \sum\limits_{k_1,k_2,k_3=1}^L
 \frac{Z_\nu}{\|{x} -a_\nu (k_1,k_2,k_3)\|}, 
\quad x\in \Omega_0.
\end{equation}

Now the projected tensor sum can be computed by a simple modification of (\ref{eqn:EwaldTensorM0}).
 
\begin{lemma}\label{lem:sumCaseP}
The projected tensor of $v_{\Omega_L} $ for the full sum over $M_0$ charges 
can be presented by rank-$(M_0 R)$ canonical tensor.
The computational cost is estimated by $O(M_0 R n L)$, while 
the storage size is bounded by $O(M_0 R  n)$.
\end{lemma}
\begin{proof}
We fix index $\nu=1$ in (\ref{eqn:EwaldSumP}) and chose the central cell 
$\Omega_0$ as above to obtain
\begin{equation}\label{eqn:EwaldSumLoc}
v_{\Omega_L} (x)=   \sum\limits_{k_1,k_2,k_3=1}^L
\frac{Z_\nu}{\|{x} -a_\nu (k_1,k_2,k_3)\|}, \quad x\in  \Omega_0,
\end{equation}
for the local lattice sum on the index set $n\times n \times n$, and
\[
 {\bf P}_{\Omega_0}= Z_\nu \sum\limits_{k_1,k_2,k_3=1}^L  {\cal W}_{\nu({\bf k})}  {\bf P}_{\Omega_0}
= Z_\nu \sum\limits_{k_1,k_2,k_3=1}^L \sum\limits_{q=1}^{R_{\cal N}}
{\cal W}_{\nu({\bf k})} {\bf p}^{(1)}_{q} \otimes {\bf p}^{(2)}_{q} \otimes {\bf p}^{(3)}_{q}
\in \mathbb{R}^{n\times n \times n},
\]
for the corresponding local projected tensor of small size $n\times n \times n$.
Here we adapt the $\Omega$-windowing operator,
$
{\cal W}_{\nu({\bf k})}={\cal W}_{\nu(k_1)}^{(1)}\otimes {\cal W}_{\nu(k_2)}^{(2)}
\otimes {\cal W}_{\nu(k_3)}^{(3)},
$
that projects onto the small $n\times n \times n$ unit cell by shifting on  the lattice vector 
${\bf k}=(k_1,k_2,k_3)$. Now the canonical representation follows by the arguments as
in the proof of Theorem \ref{thm:sumCaseE}, with the similar complexity analysis.
\end{proof}

Figure \ref{fig:3DStructCanVect} shows the assembled canonical vectors for 
a lattice structure in a periodic setting. 

\begin{figure}[htbp]
\centering
\includegraphics[width=7.0cm]{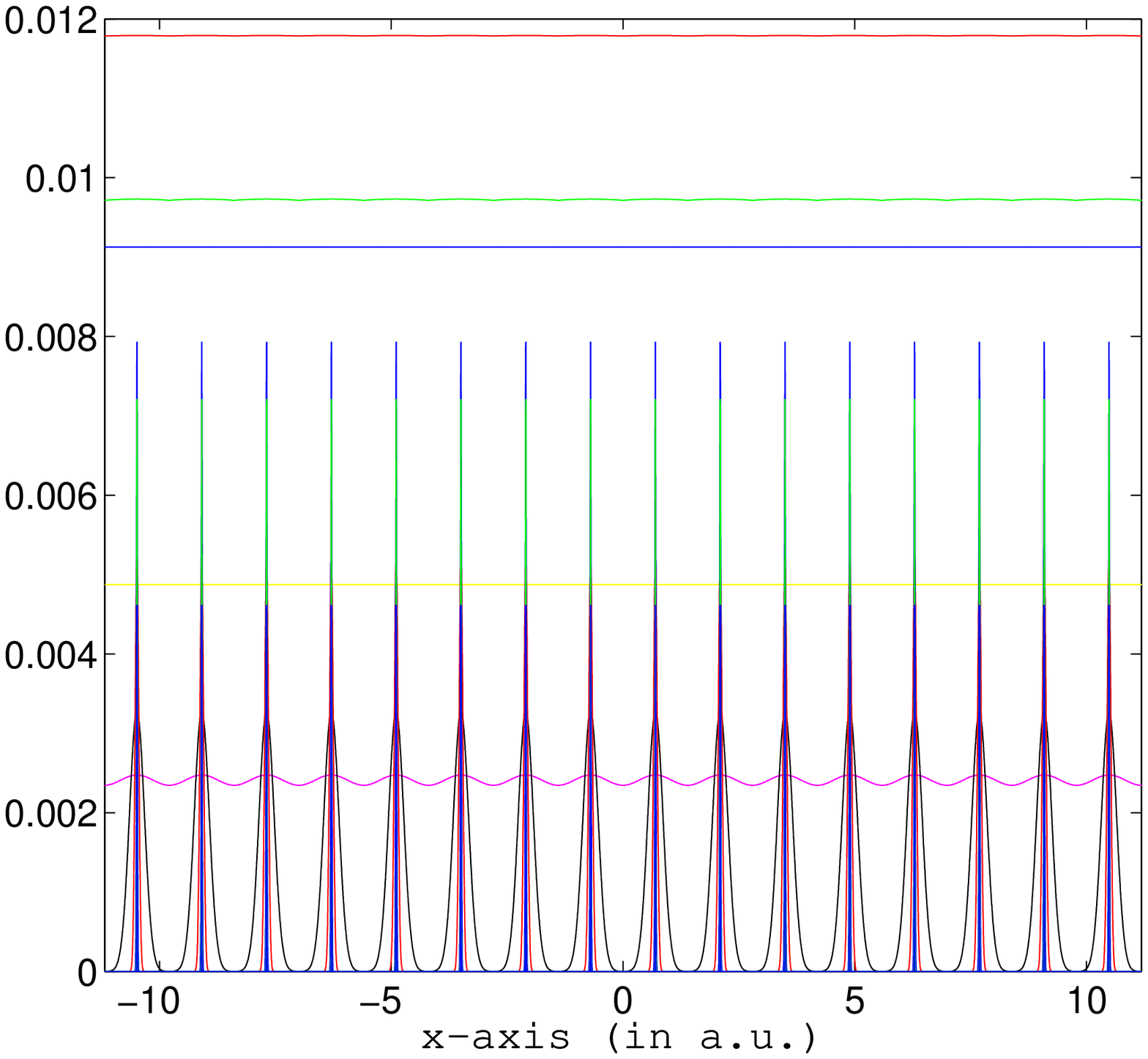}\quad \quad
\includegraphics[width=7.0cm]{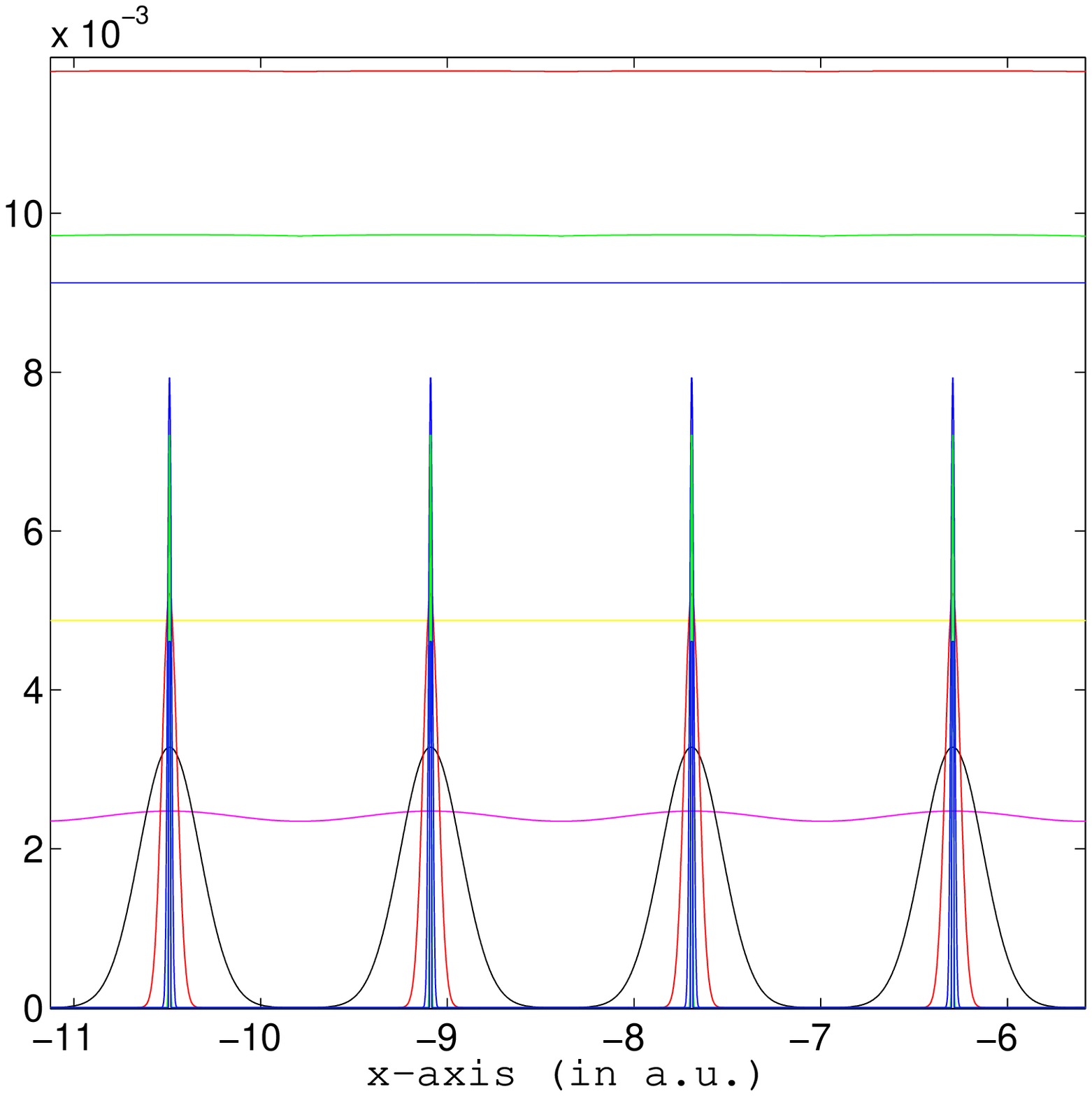}
\caption{Periodic canonical vectors in the $L\times 1\times 1$ lattice sum, $L=16$ (left); 
Zooming of four periods (right).}
\label{fig:3DStructCanVect}  
\end{figure}

Recall that in the limit of large $L$ the lattice sum ${\bf P}_{c_L}$ of the Newton 
kernels is known to converge  only conditionally. The same is true for a sum in a box.
The maximum norm increases as $C_1 \log L$, $C_2 L$ and $C_3 L^2$ 
for 1D, 2D and 3D sums, respectively (see Figure\ref{fig:ConditSum}).
This issue is of special significance in the periodic setting, dealing with the limiting case
$L \to \infty$.
To approach the limiting case, we compute a ${\bf P}_{c_L}$ on a sequence of 
large parameters $L, 2L, 4L$ etc. and then apply the Richardson extrapolation as 
described in the following. As result, we obtain the regularized tensor $\widehat{p}_L$ 
restricted to the reference unit cell $\Omega_0$.


Figure \ref{fig:ConditSum} presents the value $p_0$ of the potential sum  at the center of 
a bounded box vs. $L$, for $L\times 1\times 1$, $L\times L\times 1$ 
and $L\times L\times L$ lattice sums, where $L=2,4,8,...,128$. 
The predicted asymptotic behaviour in $L$ is easily seen.

\begin{figure}[htbp]
\centering
\includegraphics[width=5.0cm]{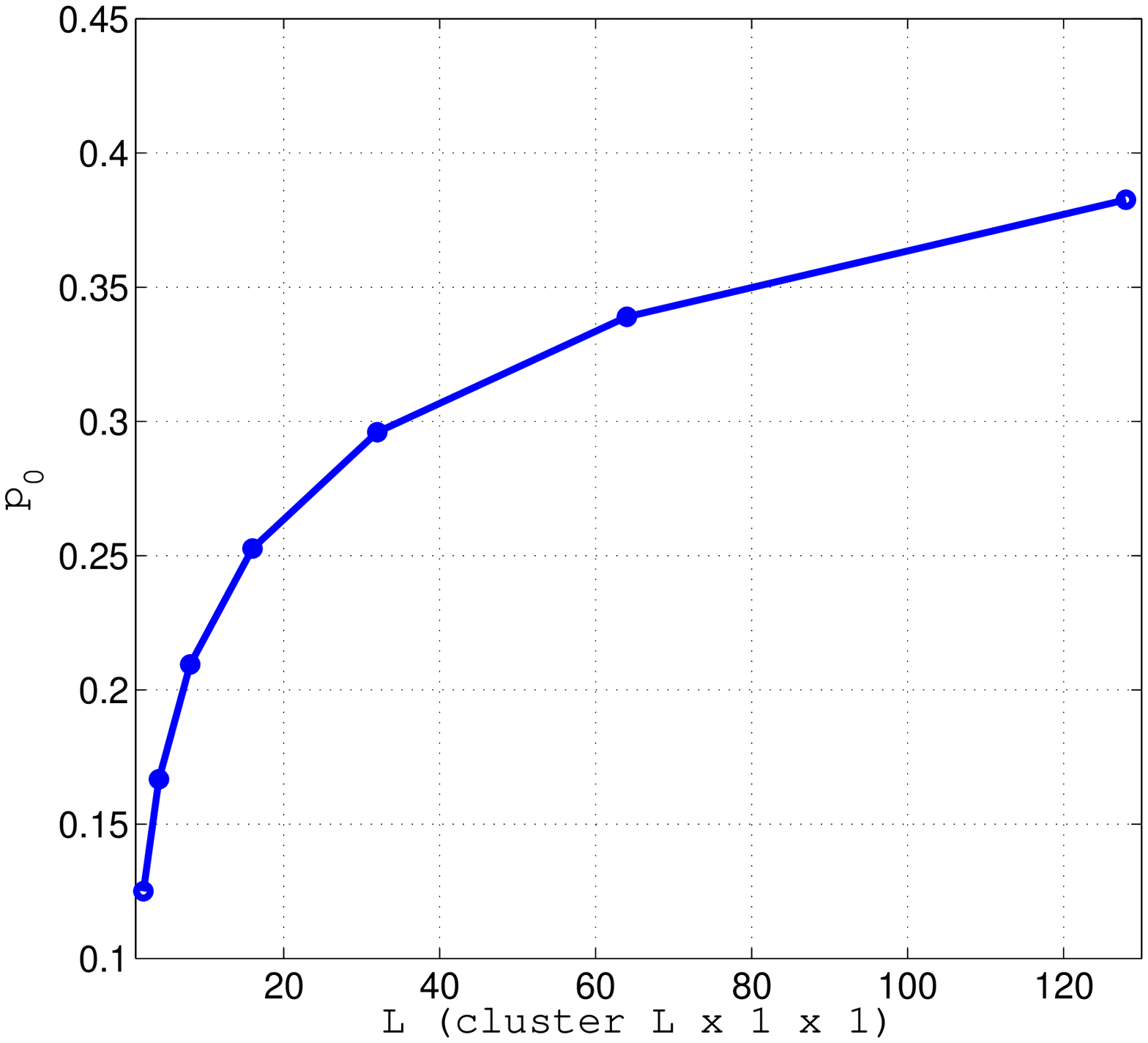}\quad 
\includegraphics[width=5.0cm]{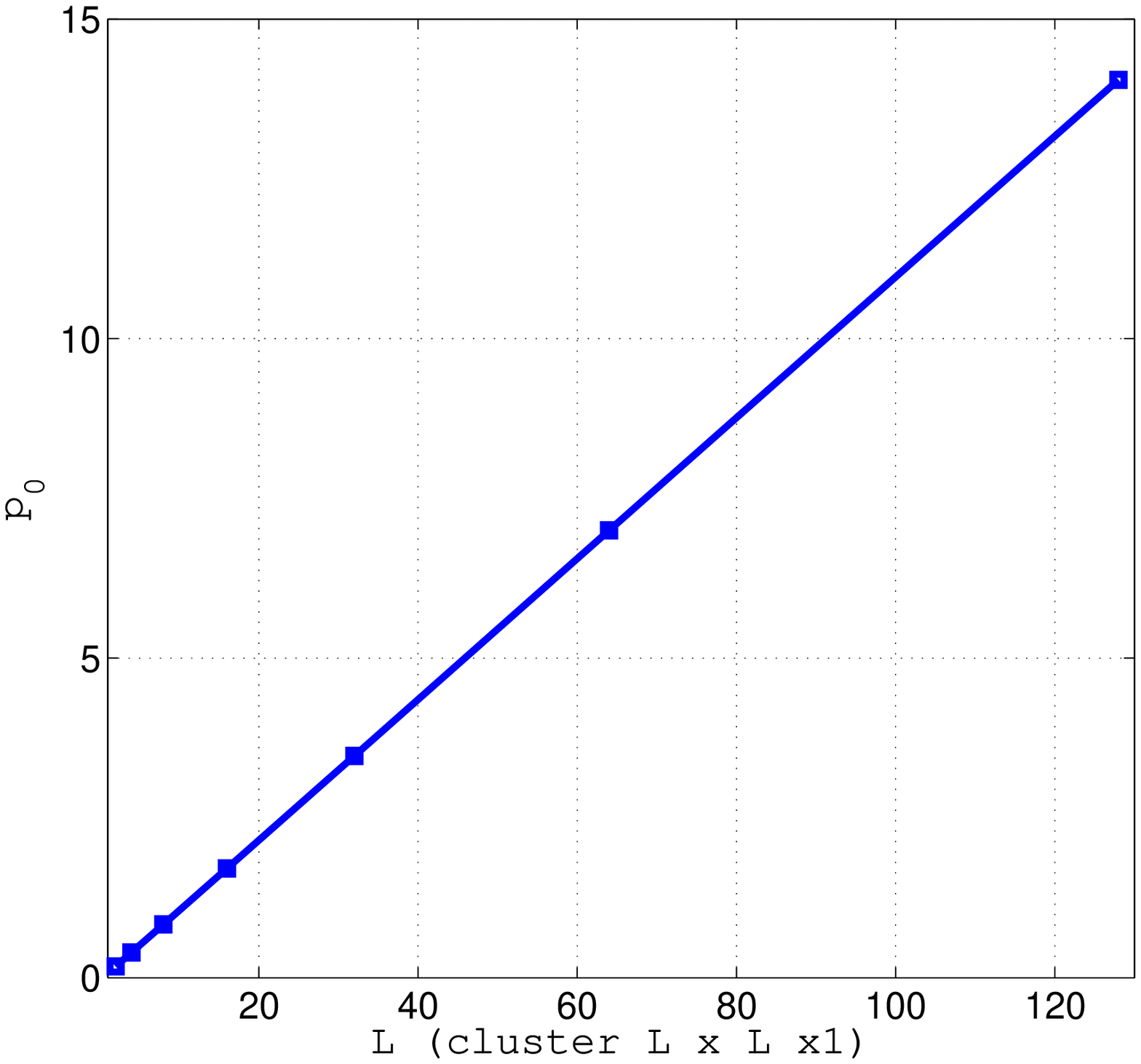}\quad 
\includegraphics[width=5.0cm]{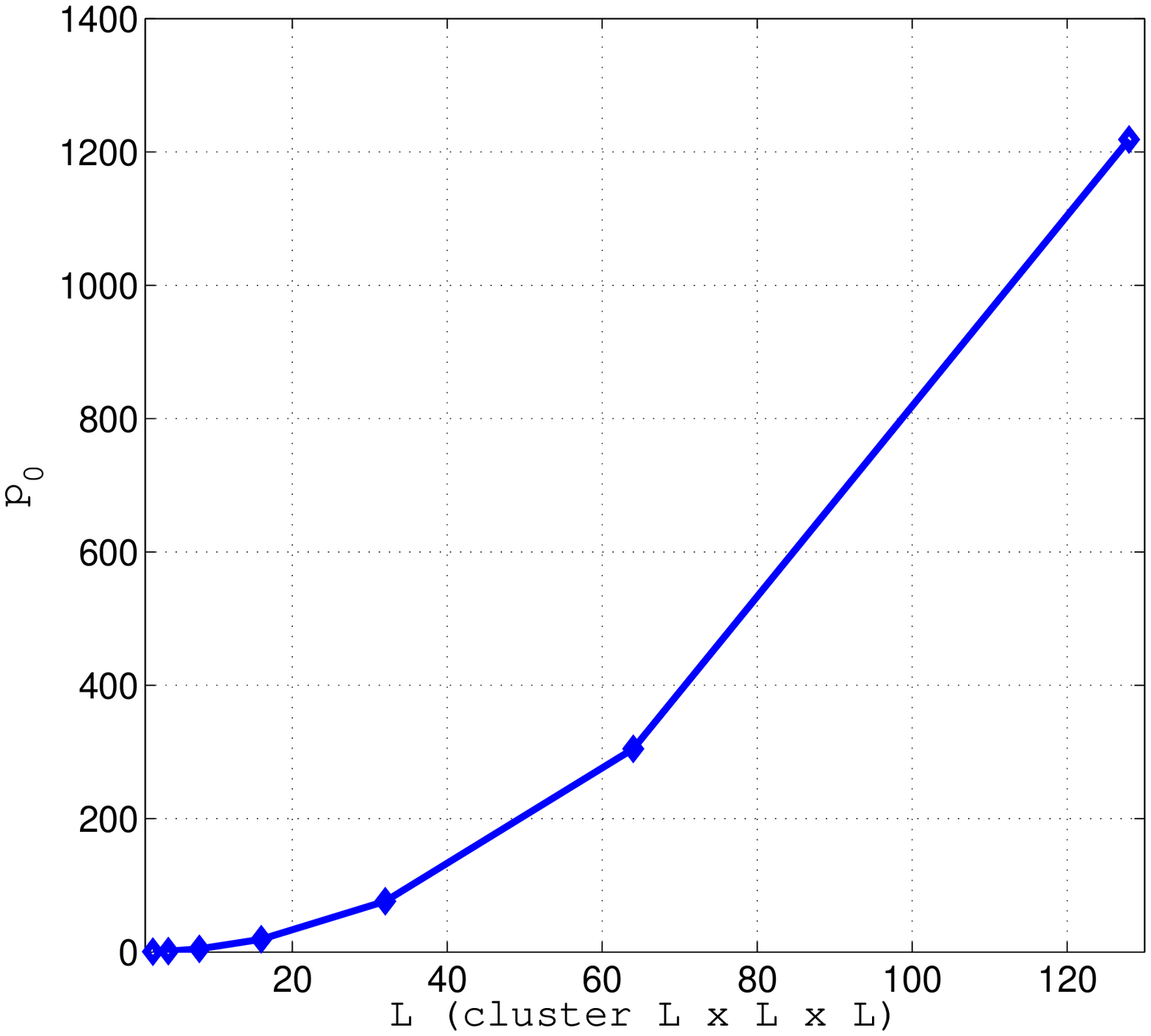}
\caption{Potential sum $p_L$ at the center of the supercell vs. $L$ for $L\times 1\times 1$, $L\times L\times 1$ 
and $L\times L\times L$ lattice sums.}
\label{fig:ConditSum}  
\end{figure}

In the traditional Ewald-type summation techniques the regularization of lattice sums
is implemented by subtraction of the analytically precomputed constants describing the asymptotic behaviour 
in $L$. In our tensor summation method this problem is solved by algebraic approach
by using the Richardson extrapolation techniques applied on a sequence of supercells with increasing size 
$L$, $2L$, $4L$, etc. Denoting the target value
of the potential by $p_L$, the extrapolation formulas for the linear and quadratic behaviour take form 
$$
\widehat{p}_L := 2p_L -p_{2L}, \quad  \mbox{and} \quad \widehat{p}_L := (4p_L -p_{2L})/3,
$$
respectively.
The effect of Richardson extrapolation is illustrated in Figure \ref{fig:CondSum_Regul}.
 
\begin{figure}[htbp]
\centering
\includegraphics[width=5.0cm]{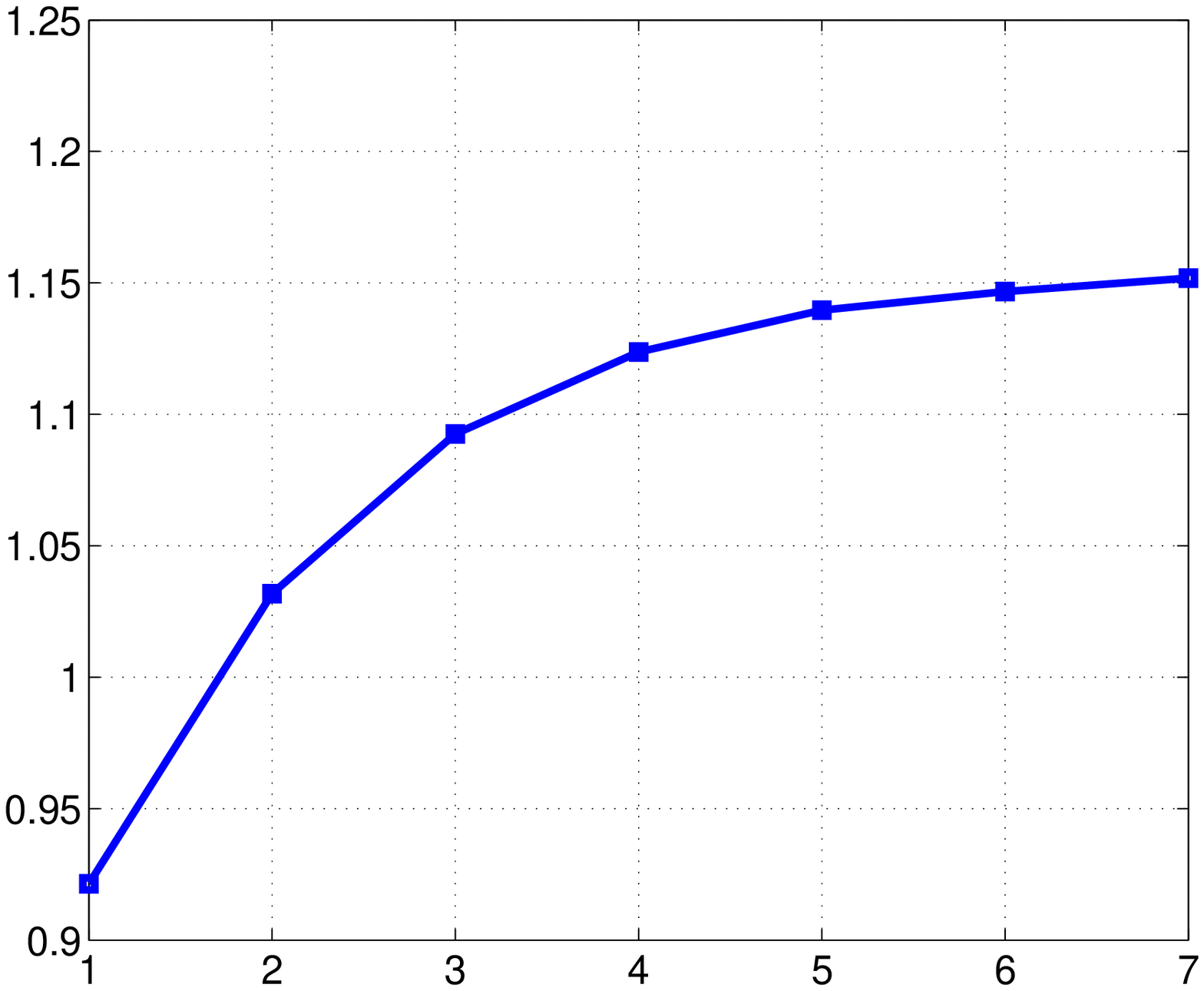}\quad
\includegraphics[width=5.0cm]{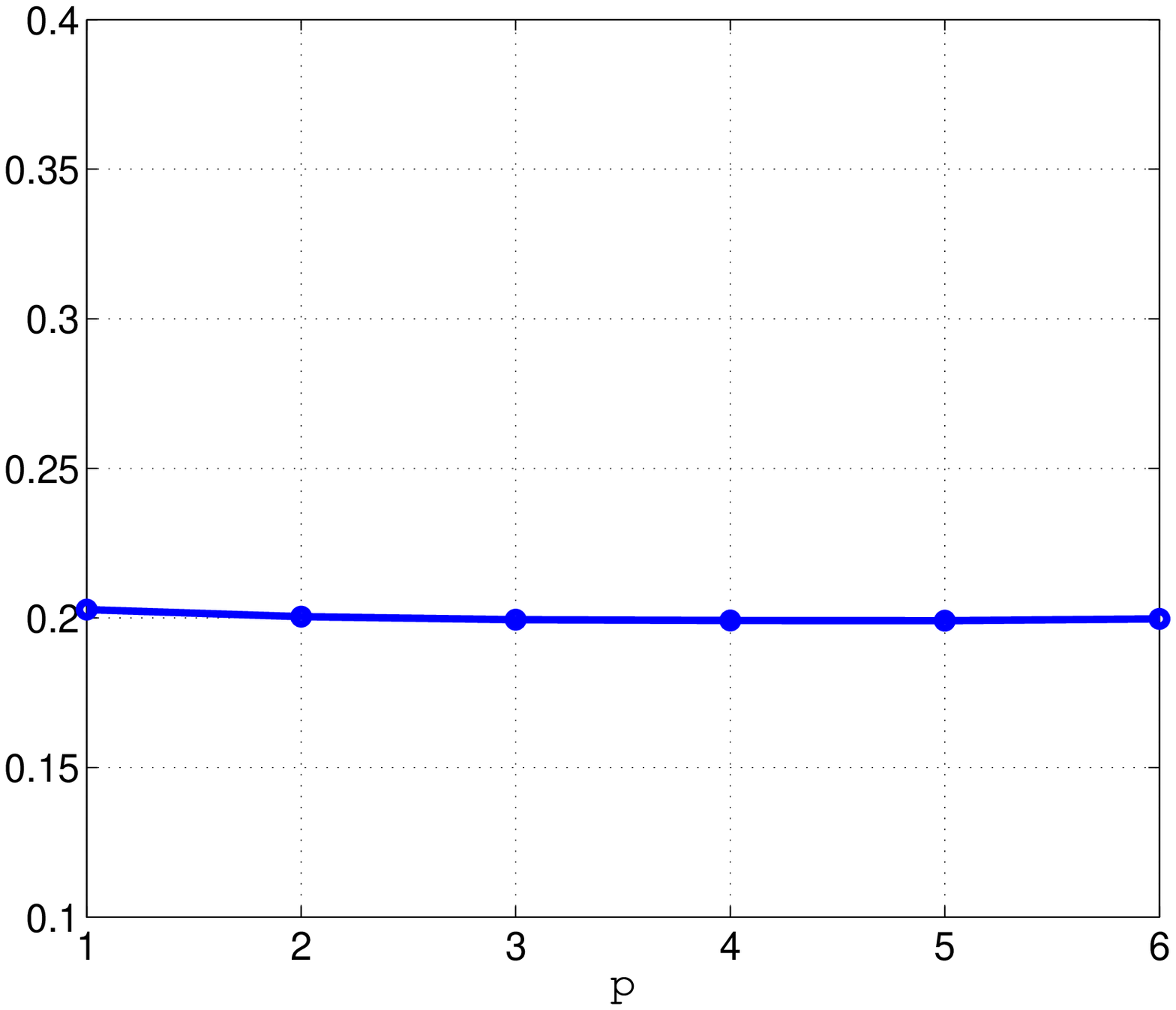}
\caption{Regularized potential sum $\widehat{p}_L$  vs. $m$ with $L=2^m$, for $L\times L\times 1$ 
(left), and $L\times L\times L$ lattice sums (right).
}
\label{fig:CondSum_Regul}  
\end{figure}
Figure \ref{fig:CondSum_Regul} indicates that the potential sum computed at the same 
point as for the previous example (in the case of $L\times L\times 1$ and 
$L\times L\times L$ lattices) 
converges to the limiting values of $\widehat{p}_L$ after application the Richardson 
extrapolation (regularized sum).

\section{ QTT ranks of the assembled canonical vectors in the lattice sum}
\label{sec:QTT-Sum}

Agglomerated canonical vectors in the rank-$R$ tensor 
representation (\ref{eqn:EwaldTensorGl}) are defined over large uniform grid of size $N_L$.
Hence numerical cost for evaluation of each of these $3 R$ vectors 
scales as $O(N_L L)$, which might become too expensive for large $L$ (recall that $N_L=n L$
scales linear in $L$). Using quantics-TT (QTT) approximation \cite{KhQuant:09}, 
this cost can be reduced to 
the logarithmic scale in $N_L$, while the storage need will become $O(\log N_L)$ only.

Our QTT-rank estimates are based on three main ingredients: the global canonical tensor
representation of $1/\|x\|$, $x\in \mathbb{R}^3$, on a supercell \cite{HaKhtens:04I,BeHaKh:08}, 
as in (\ref{eqn:Newt_canTens}), 
QTT approximation to the Gaussian (Proposition \ref{prop:gauss_rank})
and the new result on the block QTT decomposition (Lemma \ref{lem:blockQTT} below).

The next statement presents the QTT-rank estimate 
for Gaussian vector obtained by uniform sampling of $\mathrm{e}^{-\frac{x^2}{2p^2}}$
on the finite interval \cite{DoKhOsel:11}.

\begin{proposition} \label{prop:gauss_rank}
 Suppose uniform grid points $-a = x_0 < x_1 < \cdots < x_{N}=a$, $x_i = -a+hi$, $N=2^{L}$ 
are given on an interval $[-a,a]$, and the vector $G=[g_i]\in \mathbb{R}^N$ is defined by its elements 
$g_i = \mathrm{e}^{-\frac{x_i^2}{2p^2}}$, $i=0,...,N-1$. 
For given $\varepsilon > 0$, assume that $\mathrm{e}^{-\frac{a^2}{2p^2}} \leq {\varepsilon}$.
Then  there exists the QTT approximation $G_r$
of the accuracy $
 ||G - G_r||_\infty \le c \varepsilon,
 $ 
with the ranks bounded by
 $$
 rank_{QTT}(G_r) \le c \log (\frac{p}{\varepsilon}),
 $$
where $c$ does not depend on $a$, $p$, $\varepsilon$ or $N$.
\end{proposition}
\begin{proof}
The result follows by a combination of Lemma 2 and Remark 3 in \cite{DoKhOsel:11}. In fact,
the condition $\mathrm{e}^{-\frac{a^2}{2p^2}} \leq {\varepsilon}$ implies the relation 
\begin{equation}\label{eqn:GaussInterv}
a \geq a_\varepsilon =\sqrt{2}p \log^{1/2}(1/\varepsilon).
\end{equation}
Combining (\ref{eqn:GaussInterv}) and the rank-$r$ truncated Fourier series representation 
$G_r$ leads to the error bound
\[
 ||G - G_r||_\infty \le c\left(1+ \frac{1}{p}\sqrt{\log \frac{p}{\varepsilon(1+a)}}\right) 
 \varepsilon.
\]
Hence, the result follows by substitution $\varepsilon \mapsto \frac{\varepsilon}{p}$.
\end{proof}

Next Lemma proves the important result that the QTT rank of a weighted sum of 
regularly shifted bumps (see Fig. \ref{fig:Wann_type}) does not exceed the product 
of QTT ranks of the individual sample and the weighting factor.  
\begin{lemma} \label{lem:blockQTT}
 Let $N=2^L$ with the exponent $L=L_1 + L_2$, where $L_1,L_2\geq 1$, and assume that the index 
set $I:=\{1,2,...,N\}$ is split into $n_2=2^{L_2}$ equal non-overlapping subintervals 
$I= \cup_{k=1}^{n_2} I_k$,  each of length $n_1=2^{L_1}$.
Given $n_1$-vector ${\bf x}_0$ that obeys the rank-$r_0$ QTT representation, define 
$N$-vectors ${\bf x}_k$, $k=1,...,L_2$,
\begin{equation} \label{eqn:}
{\bf x}_k(i) = \bigg\{ 
\begin{array}{ll}
\ {\bf x}_0(:) \quad \       \operatorname{for} \,\, i\in I_k \\
\ 0   \quad    \ \ \ \ \ \, \operatorname{for} \,\, i\in I\setminus I_k,
\end{array}
\end{equation} 
and denote ${\bf x} ={\bf x}_1 + ... + {\bf x}_{L_2}$.
Then for any choice of $N$-vector $F$, 
we have
\[
 rank_{QTT}(F \odot {\bf x})\leq  rank_{QTT}(F)\, r_0.
\]
\end{lemma}
\begin{proof}
Since all vectors  ${\bf x}_k$ ($k=1,...,L_2$) have non-intersecting supports, $I_k$, the $L_2$-level
block quantics representation of ${\bf x}$ (see \cite{KhQuant:09}) becomes separable and, 
we obtain the separable decomposition
\[
 {\cal Q}_L({\bf x}_1 + ... + {\bf x}_{L_2})= (\otimes_{k=1}^{L_2} {\bf 1}) \otimes 
 {\cal Q}_{L_1}({\bf x}_0),\quad {\bf 1}=(1,1)^T,
\]
resulting in the rank bound
\[
 rank_{QTT}({\bf x})\leq r_0.
\]
Combining this bound with the standard rank estimate for Hadamard product of tensors
completes the proof.
\end{proof}

\begin{remark} Lemma \ref{lem:blockQTT} provides the constructive algorithm and 
rigorous proof of the low QTT-rank decomposition for certain class of Bloch 
functions \cite{Bloch:1925} and Wannier-type functions.
\end{remark}

\begin{figure}[htbp]
\centering
\includegraphics[width=7.0cm]{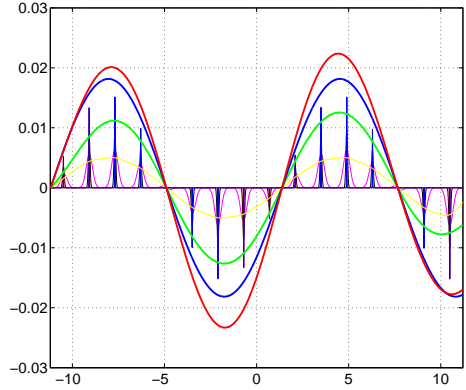} 
\caption{Canonical vectors of the lattice sum modulated by a $\sin$-function.}
\label{fig:Wann_type}  
\end{figure}

Figure \ref{fig:Wann_type} illustrates shapes of the assembled canonical vectors
modulated by a $\sin$-harmonics imitating the construction
of the Wannier-type functions.

Now we are able to estimate QTT ranks of the assembled canonical
vectors representing the lattice sum. In this study, we analyze the canonical decomposition based
on the initial (non-optimized) quadrature (\ref{eqn:Newt_canTens}),
where each term is obtained by sampling of a Gaussian on the uniform 3D grid.
In practice, we apply the optimized quadrature obtained from  the previous one by certain
algebraic rank reduction \cite{BeHaKh:08}. This optimization procedure slightly modifies
the shape of canonical vectors, however, the numerical tests indicate merely the same
QTT ranks as predicted by our theory for the Gaussian-type vectors.

\begin{lemma} \label{lem:rQTTLatSum}
For given tolerance $\varepsilon >0$,
suppose that the set of Gaussian functions $S:=\{g_k= e^{- t_k^2 \|x\|^2} \}$, 
$k=0,1,...,M$, representing canonical vectors in tensor decomposition $\mathbf{P}_R$, 
is specified by parameters in (\ref{eqn:hM}). 
Let us split the set $S$ into two subsets $S=S_{loc}\cup S_{glob}$, such that  
$$
S_{loc}:=\{g_k: a_\varepsilon(g_k) \leq b\}\quad  \mbox{and} \quad S_{glob}=S\setminus S_{loc}.
$$
where $a_\varepsilon(g_k)$ is defined by (\ref{eqn:GaussInterv}).
Then 
the QTT-rank of each canonical vector $v_q$, $q=1,...,R$, in (\ref{eqn:EwaldTensorGl}), 
where $R=M+1$, corresponding to $S_{loc}$ obeys the uniform in $L$ rank bound
\[
 r_{QTT}\leq C \log(1/\varepsilon).
\]
For vectors in $S_{glob}$ we have the rank estimate
\[
 r_{QTT}\leq C \log(L/\varepsilon).
\]
\end{lemma}
\begin{proof}
In our notation we have $1/(\sqrt{2}p_k) =t_k=(k \log M)/M$, $k=1,...,M$ ($k=0$ is the trivial case).
We omit the constant factor $\sqrt{2}$ to  obtain $p_k= M/(k \log M)$.

For functions $g_k\in S_{loc}$, the relation (\ref{eqn:GaussInterv}) implies
\[
O(1) =b \geq a_\varepsilon(g_k) = \sqrt{2} p_k \log^{1/2}(1/\varepsilon),
\]
implying the uniform bound $p_k \leq C$, and then the rank estimate $r_{QTT}\leq C \log(1/\varepsilon)$
in view of Proposition \ref{prop:gauss_rank}.
Now we apply Lemma \ref{lem:blockQTT} to obtain the uniform in $L$ rank bound. 

For globally supported functions in $S_{glob}$ we have 
$bL \geq a_\varepsilon \simeq p_k \log^{1/2}(1/\varepsilon) \geq b $, 
hence we will consider all these function on the maximal support of the size of supercell,  $bL$,
and set $a=bL$.
Using the trigonometric representation as in the proof of Lemma 2 in \cite{DoKhOsel:11},
we conclude that for each fixed $k$ the shifted Gaussians, 
$g_{k,\ell}(x)= e^{- t_k^2 \|x- \ell b\|^2}$ ($\ell=1,...,L$), 
can be approximated by shifted trigonometric series 
$$
 G_r(x-b\ell) = \sum\limits_{m=0}^M C_m p\mathrm{e}^{-\frac{\pi^2 m^2 p^2}{2a^2}} 
\cos\left(\dfrac{\pi m (x-b\ell)}{a}\right),\quad a=bL,
 $$
which all have the common trigonometric  basis containing about 
$rank_{QTT}(G_r)=O(\log (\frac{p_k}{\varepsilon}))= O(\log (\frac{bL}{\varepsilon}))$ terms. Hence the sum 
of shifted Gaussian vectors over
$L$ unit cells will be approximated with the same QTT-rank bound as each individual 
term in this sum, which proves the assertion. 
\end{proof}

\begin{figure}[htbp]
\centering
\includegraphics[width=7.6cm]{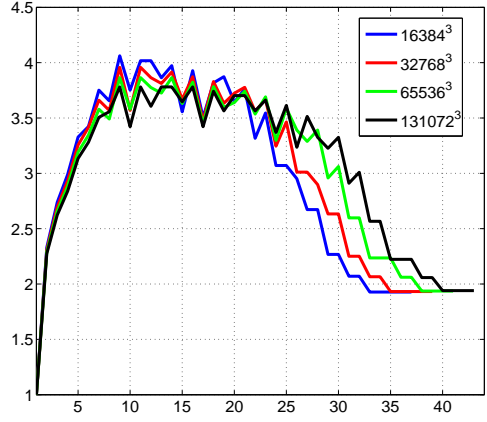}  
\caption{ QTT-ranks of the canonical vectors of a single 3D Newton kernel
discretized on a cubic grids of size $n^3=16384^3,\, 32768^3,\, 65536^3$ and $131072^3$.
}
\label{fig:QTTRanks_Newton}  
\end{figure}
Based on the previous statements, we arrive at the following result.
\begin{theorem}\label{thm:sumQTT}
 The projected tensor of $v_{c_L} $ for the full sum over a single  charge 
can be presented by the rank-$R$  QTT-canonical tensor 
\begin{equation}\label{eqn:EwaldTensorQTT}
 {\bf P}_{c_L}=  \sum\limits_{q=1}^{R}
({\cal Q}\sum\limits_{k_1=1}^L{\cal W}_{\nu({k_1})} {\bf p}^{(1)}_{q}) \otimes 
({\cal Q}\sum\limits_{k_2=1}^L {\cal W}_{\nu({k_2})} {\bf p}^{(2)}_{q}) \otimes 
({\cal Q}\sum\limits_{k_3=1}^L{\cal W}_{\nu({k_3})} {\bf p}^{(3)}_{q}),
\end{equation}
where the QTT-rank of each canonical vector is bounded by 
$r_{QTT}\leq C \log(L/\varepsilon) $.
The computational cost is estimated by $O(R  L r_{QTT}^3)$, 
while the storage size scales as $O(R \log^2(L/\varepsilon))$.
\end{theorem}

Figure \ref{fig:QTTRanks_Newton} represents QTT-ranks of the canonical vectors of 
a single 3D Newton kernel discretized on a large cubic grids. 

Figure \ref{fig:QTTRanks_Ewald} demonstrates that the average QTT ranks of the 
assembled canonical vectors
for $k=1,...,R$, scale logarithmically both in $L$ and in the total grid-size $n=N_L$. 
\begin{figure}[htbp]
\centering
\includegraphics[width=6.0cm]{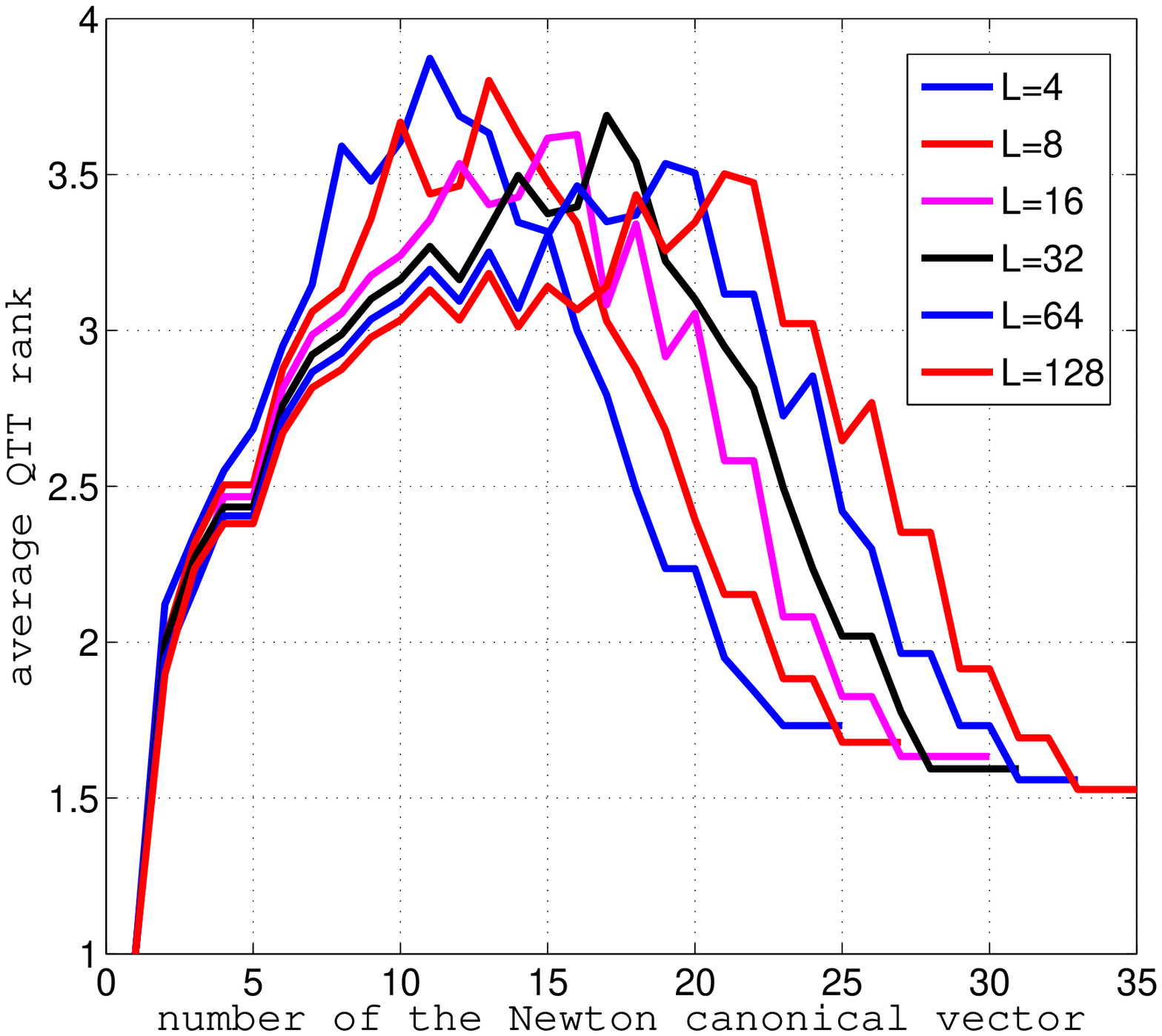} \quad \quad
\includegraphics[width=5.7cm]{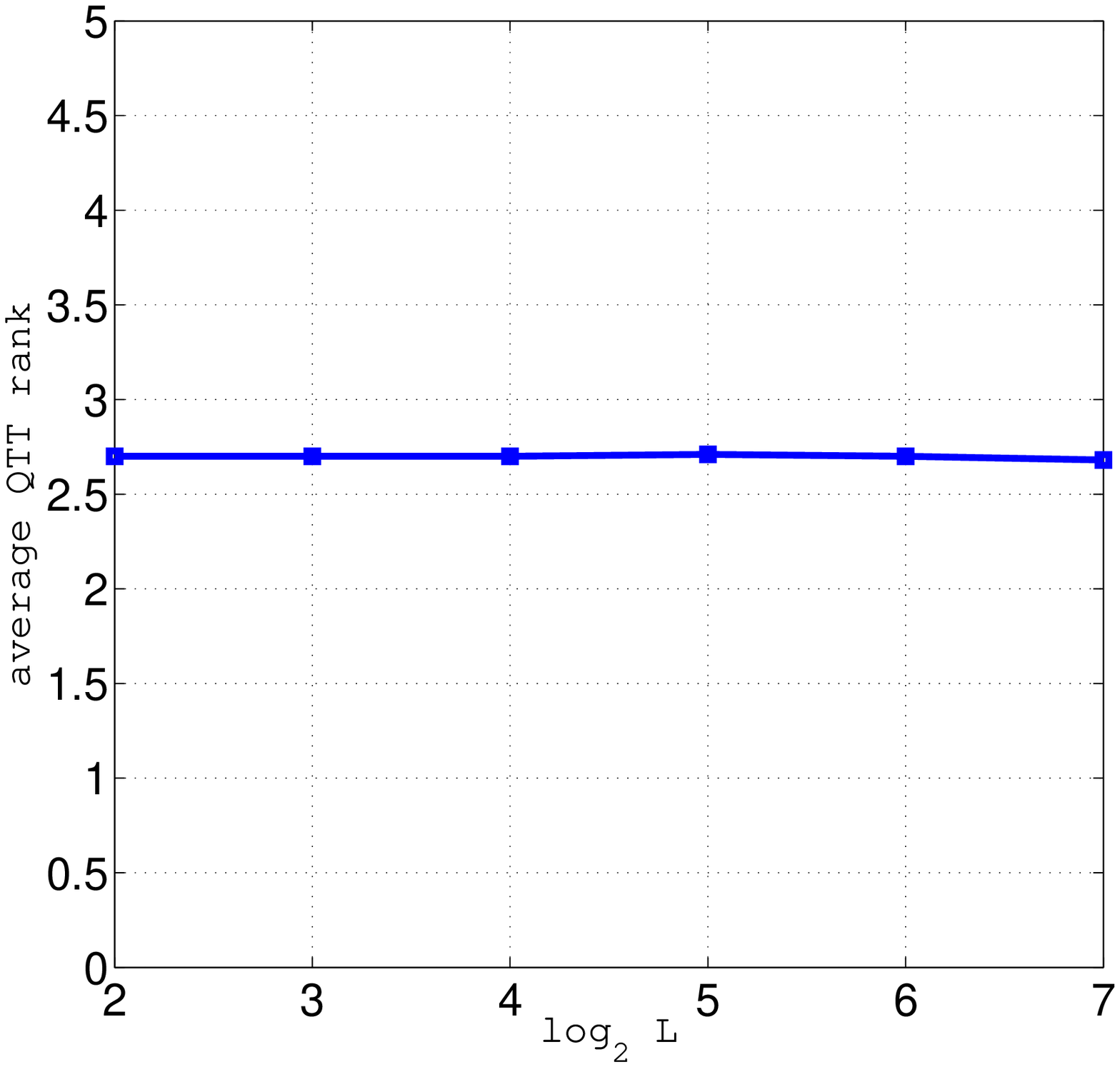}
\caption{Left: QTT ranks of the assembled canonical vectors vs. $L$ for fixed grid 
size $N^3=16384^3$. Right: Average QTT-ranks over $R$ canonical vectors 
vs. $\log L$ for 3D  evaluation of the $L\times 1\times 1$ chain of Hydrogen atoms 
on $N \times N\times N$ grids, $N=2048,\, 4096,\, 8192$, $16384$.
}
\label{fig:QTTRanks_Ewald}  
\end{figure}

\section{Conclusions}\label{sec:Conclusions}

We introduce the method of assembled rank-structured calculation of 
large $L\times L \times L$ lattice sums of potentials 
 discretized on $N\times N \times N$ 3D Cartesian grid 
in a box and for the supercell in periodic boundary conditions.

Advantages of the tensor approach applied to the lattice summation problem
are achieved due to combination of two basic concepts. First, 
we apply the low-rank separable tensor-product representation  
of a single Newton kernel discretized on a fine $N\times N \times N$ spacial 
grid, that can be used as the master tensor for the translation along any lattice
vector. Second, the global 3D product-type geometry in the location of interaction potentials
enables us to employ the assembled summation of shifted low-rank canonical tensors.
The latter allows to adapt tensor calculus to reduce the 3D summation to 1D sums 
of $N$-vectors. 

The total lattice sum of shifted potentials in a  box is proven to preserve 
the same canonical rank as that for a single Newton kernel. 

For the case of lattice sums in a box our approach exhibits   linear scaling 
in $L$, for both computational work and storage size, reducing dramatically
the numerical costs compared with other summation methods.
For example, computation of a sum of $10^6$ electrostatic potentials 
of Hydrogen nuclei in a box takes about $2$ seconds, when using our algorithms
implementation in Matlab  on a terminal of 8 AMD Opteron cluster (see Table \ref{Table_timesL}). 
Comparison of the  direct canonical sums of the electrostatic potentials 
and the assembled lattice tensor summation demonstrates the accuracy 
at the level of machine precision, $10^{-14}$.

In the periodic setting, the storage size is uniformly bounded in $L$. 

For both models, we prove that QTT approximation
method reduces the complexity to logarithmic scaling in the total grid size, $O(\log N)$. 
This suggests the efficient approach to numerical simulations on large 
$L\times L \times L$ lattices since the limitations on the spacial grid-size are essentially 
relaxed.

It is worth to note that the sum of electrostatic potentials is calculated in a whole 
computational box/supercell  in a convenient structured form, which is suitable for further  
numerical treatment of the 3D quantities involved by using tensor methods in 1D complexity, 
for example, integration,  differentiation, multiplication with a function, etc. 

This approach can be also applied to a wide class of commonly used chemical potentials, 
in particular, to Coulomb-type, Yukawa, Helmholtz, Slater, Stokeslet, 
Lennard-Jones or van der Waals interactions. 
In all these cases the low-rank tensor decomposition
can be proved to exist and can be constructed by the analytic-algebraic methods 
as in the case of Newton kernel.

\section{Appendix: Basics of rank-structured tensor formats and operations}\label{sec:Append}

Separable representation of the multidimensional arrays in the Tucker and canonical tensor formats, 
were since long known in the computer science community \cite{Kolda}, where they were
 mostly used in  processing of the multidimensional experimental data in chemometrics, 
psychometrics and in  signal processing.
The remarkable approximating properties of the Tucker and canonical decomposition 
for wide classes of function related tensors  were revealed in \cite{Khor:06,KhKh:06}, 
promoting its usage as a tool for the numerical treatment of the multidimensional problems in 
numerical analysis. An introductory description of tensor formats for 
function related tensors and tensor-structured numerical methods for calculation of multidimensional
functions and operators is presented in \cite{VeKh_Diss:10,KhorSurv:10}.

A tensor is a multidimensional array given by a $d$-tuple index set,
\[
{\bf A}=[a_{i_1,...,i_d}] \in \mathbb{R}^{n_1 \times \ldots \times n_d},
\quad i_\ell \in \{ 1,\ldots ,n_\ell\}.
\]
It is an element of a linear vector space equipped with the Euclidean scalar product.
For tensor with equal sizes $n_\ell =n$, $\ell=1,\ldots d$, the required storage 
is $n^{\otimes d}$. 
To get rid of the exponential growth of the tensor with the dimension $d$, one can
employ the rank-structured representations of the multidimensional arrays.

As a building block for such representation we use a rank-$1$ tensor, which is a tensor product
of vectors in each dimension,
\[
{\bf A} = {\bf a}^{(1)}\otimes ... \otimes {\bf a}^{(d)}\in \mathbb{R}^{n_1 \times \ldots \times n_d}
\]
with entries $ a_{i_1,\ldots i_d}=a^{(1)}_{i_1}\cdot \cdot\cdot a^{(d)}_{i_d}$.
 
Taking a sum of $R$  rank-$1$ tensors with some weights $c_k$ one comes to the 
canonical rank-$R$ representation,
 \begin{equation}\label{CP_form}
{\bf A} = \sum\limits_{i =1}^{R} c_i {\bf a}_i^{(1)}  \otimes \ldots \otimes {\bf a}_i^{(d)},
 \quad  c_i \in \mathbb{R},
\end{equation}
where ${\bf a}_i^{(\ell)}$ are normalized vectors. The tensor in the canonical format 
requires storage $O(d n R)$.  Sometimes ${\bf a}_i^{(\ell)}$, 
$\ell=1,2,\; \ldots ,d$, are called the ''skeleton`` vectors of the canonical rank-$R$
tensor representation, while the matrices $A^{(\ell)} =[{\bf a}_1^{(\ell)}, \ldots ,{\bf a}_R^{(\ell)}]$ 
obtained by sticking together all vectors of the same mode ${\bf a}_i^{(\ell)}$, ($i=1,2,\ldots, R$) 
are called ''factor matrices'' of the canonical tensor.

The Tucker decomposition is constructed using the orthogonal projection of the original
tensor by the orthogonal matrices. It is also a sum of the tensor products,
\[
  {\bf A} ={\sum}_{\nu_1 =1}^{r_1}\ldots
{\sum}^{r_d}_{{\nu_d}=1} \beta_{\nu_1, \ldots ,\nu_d}
\,  {\bf a}^{(1)}_{\nu_1} \otimes \ldots \otimes {\bf a}^{(d)}_{\nu_d},\quad \ell=1,\ldots,d ,
\]
where ${\bf r}=(r_1,...,r_d)$ is the Tucker rank, $\boldsymbol{\beta}=
[\beta_{\nu_1,...,\nu_d}]$ is the core
tensor, and the set of orthonormal vectors ${\bf a}^{(\ell)}_{\nu_\ell}\in \mathbb{R}^{n_\ell}$,
form the orthogonal matrices of the Tucker projection.

The rank-structured tensor representation provides 1D complexity of multilinear operations 
with multidimensional tensors. 
Rank-structured tensor representation  provides fast multi-linear algebra with linear 
complexity scaling in the dimension $d$. 

For given canonical tensors ${\bf A}$ and ${\bf B}$ 
with the ranks $R_a$ and $R_b$, respectively, their Euclidean scalar product can be computed by
\begin{equation}\label{scal}
\left\langle {\bf A}, {\bf B} \right\rangle:=
\sum\limits_{i=1}^{R_a} \sum\limits_{j=1}^{R_b}
c_i c_j \prod\limits_{\ell=1}^d \left\langle  {\bf a}_i^{(\ell)},  {\bf b}_j^{(\ell)} \right\rangle,
\end{equation}
at the expense $O(d n R_a R_b)$.
The Hadamard product of tensors ${\bf A}$, ${\bf B}$ given in the
form (\ref{CP_form}) is calculated in $O( d n R_a R_b )$ operations by 1D point-wise
products of vectors, 
\begin{equation}\label{had}
{\bf A}  \odot {\bf B}     :=
\sum\limits_{i=1}^{R_a} \sum\limits_{j=1}^{R_b}
c_i c_j \left(  {\bf a}_i^{(1)} \odot  {\bf b}_j^{(1)} \right) \otimes \ldots \otimes
\left(  {\bf a}_i^{(d)} \odot  {\bf b}_j^{(d)} \right).
\end{equation}
Summation of two tensors in the canonical format ${\bf C} ={\bf A}  + {\bf B}$ is performed 
by a simple concatenation of their factor matrices, 
$A^{(\ell)} =[{\bf a}_1^{(\ell)}, \ldots ,{\bf a}_{R_a}^{(\ell)}]$ and
$B^{(\ell)} =[{\bf b}_1^{(\ell)}, \ldots ,{\bf b}_{R_b}^{(\ell)}]$,
\begin{equation}\label{eqn:conc}
C^{(\ell)} =[{\bf a}_1^{(\ell)}, \ldots ,{\bf a}_{R_a}^{(\ell)},{\bf b}_1^{(\ell)}, \ldots ,
{\bf b}_{R_b}^{(\ell)}].
\end{equation} 
The rank of the resulting canonical tensor is $R_c =R_a + R_b$.

In electronic structure calculations, the 3D convolution transform with the 
Newton kernel, $\frac{1}{\|x-y\|}$, is the most computationally expensive operation.
The tensor method to compute  convolution over large $n\times n\times n$  Cartesian grids
in $O(n\log n)$ complexity was introduced in \cite{Khor1:08}.\\
Given canonical tensors ${\bf A}$, ${\bf B}$ in a form (\ref{CP_form}),  
their convolution product  is represented by the sum of tensor products of $1D$ convolutions, 
\begin{equation}\label{conv}
 {\bf A}  \ast {\bf B} = \sum\limits_{i=1}^{R_a}
\sum\limits_{ j= 1}^{R_b}  c_i c_j \left(  {\bf a}_i^{(1)} \ast {\bf b}_j^{(1)} \right) 
\otimes \left(  {\bf a}_i^{(2)} \ast {\bf b}_j^{(2)} \right) 
 \otimes \left(  {\bf a}_i^{(3)} \ast {\bf b}_j^{(3)} \right) ,
\end{equation}
where ${\bf a}_k^{(\ell)} \ast {\bf b}_m^{(\ell)}$ is the convolution product of $n$-vectors.
The cost of tensor convolution in both storage and time is estimated by 
$O(R_a R_b n \log n)$. It  considerably outperforms the 
conventional  3D FFT-based algorithm of complexity $O(n^3 \log n)$ \cite{KhKh3:08}.

In tensor-structured numerical methods the calculation of the $3D$  convolution integrals  is
replaced by a sequence of $1D$  scalar and Hadamard products, and $1D$ convolution 
transforms \cite{KhKh3:08,VeKh_Diss:10}. However, the sequences of rank-structured 
operations lead to increasing of tensor ranks since they are multiplied. 
For rank reduction, for example, 
the canonical-to-Tucker and Tucker-to-canonical algorithms can be used 
\cite{KhKh:06,KhKh3:08,VeKh_Diss:10}.
 
The matrix-product states (MPS) decomposition is since long used 
 in quantum chemistry and quantum information theory \cite{White:93,Cirac_TC:04}.
The particular case of MPS representation is called a tensor train (TT) format \cite{Osel_TT:11}.
Any entry of a $d$th order tensor in this format is given by
\begin{equation}\label{eq:tt}
 a(i_1,i_2,\ldots,i_d) = A^{(1)}_{i_1} A^{(2)}_{i_2} \ldots A^{(d)}_{i_d}, 
\end{equation}
where each $A^{(k)}_{i_k}=A^{(k)}(\alpha_{k-1},i_k,\alpha_k)$ is $r_{k-1} \times r_k$ 
matrix depending on $i_k$ with the convention $r_0=r_d=1$.
Storage size for $n^{\otimes d}$ TT tensor is bounded by $O(d r^2 n)$, $r=\max {r_k}$.
The algebraic operations on TT tensors can be  
implemented with linear complexity scaling in $n$ and $d$. 

In 2009 the quantics-TT (QTT) tensor approximation method 
was introduced\footnote{B.N. Khoromskij. \emph{ $O(d\log N)$-Quantics Approximation
of $N$-$d$ Tensors in High-Dimensional Numerical Modeling.}
Preprint 55/2009, Max-Planck Institute for Mathematics in the Sciences, Leipzig 2009;
http://www.mis.mpg.de/publications/preprints/2009/prepr2009-55.html.}  
and rigorously proved to provide logarithmic scaling in storage for
a wide class of function generated vectors and multidimensional tensors, 
see also  \cite{KhQuant:09}.
In particular, the QTT representation of function-related vectors of size 
$N=q^L$, ($q=2,3,...$) needs only
\[
q\cdot L \cdot r^2 \ll q^L
\]
numbers to store, where $r$ is the QTT-rank of $q\times q \times ... \times q$ tensor 
of order $L$,
reshaped from the initial vector by $q$-adic folding  \cite{KhQuant:09}.
For example, the $N$-vector ${\bf x}=[x_i]$ of size $N=q^L$ is reshaped  to its quantics image
in $\mathbb{Q}_L :=\bigotimes_{\ell=1}^L \mathbb{R}^q$ via $q$-coding,
\[
 i-1=\sum_{\ell=1}^L (j_\ell -1)q^{\ell-1}, \quad j_\ell \in\{1,2, ...,q\}.
\]
Though the optimal choice is shown to be $q=2$ or $q=3$, the numerical implementations 
are usually performed with $q=2$ (binary coding).

In \cite{KhQuant:09} it was proven that the rank parameter $r$ in the QTT approximation  
is a small constant for a wide class of functions discretized on the uniform grid.
For example, $r=1$ for complex exponents, 
$r=2$ for trigonometric functions and  for Chebyshev polynomials 
(sampled on Chebyshev-Gauss-Lobatto grid), and
$r\leq m+1$ for polynomials of degree $m$ (see also \cite{Gras:10}). 
These properties are extended to various combinations of above functions.

The QTT approximation method enables  the  multidimensional vector transforms
with logarithmic complexity scaling, $O(\log N)$. For example, we mention
the superfast FFT \cite{DoKhSav:11}, Laplacian inverse \cite{KazKhor_1:10} 
and wavelet \cite{KhMiao:13} transforms.

\begin{footnotesize}

\end{footnotesize}

\end{document}